\documentclass[reqno,12pt]{amsart}
\usepackage{willie_style}

\usepackage{xr-hyper}
\title{The ribbon category framework for topological quantum computing}
\author[W.~Aboumrad]{Willie Aboumrad}
\address[W.~Aboumrad]{The Institute for Computational and Mathematical Engineering (ICME) at Stanford University}
\email{willieab@stanford.edu}
\urladdr{https://web.stanford.edu/~willieab}

\date{}

\begin{document}

\maketitle

\begin{abstract}
	This expository article supplies the mathematical background underpinning the braid representation calculator introduced in \cite{willie_aqc}; those representations describe the sets of logic gates available to a topological quantum computer for processing encoded qubits. Assuming little background in category theory, we first recall the notion of a ribbon fusion category (RFC), collecting most of the necessary definitions. Then we discuss how certain RFCs arise from the representation theory of quantum groups. We explore the braiding in these categories in detail, since it is essential for the quantum computing application. 
\end{abstract}	

\tableofcontents

\section{Introduction}
The notion of a \textit{ribbon fusion category} (RFC) features prominently in the story of topological quantum computers powered by anyons. This expository \chorpaper\ recalls that notion, explains how RFCs induce representations of the Artin braid group, and describes how RFCs arise from quantum groups.  We will more or less introduce all the necessary terminology. \ddichotomy{Although the material in this \chorpaper\ is well-known to experts, it is important for two reasons: first, it introduces some necessary background; second, it explains how the Casimir element in a classical universal enveloping algebra determines the braiding in the category of modules over the corresponding quantum group-- understanding that braiding is the focus of this work.}{In particular, this \chorpaper\ provides the background in category theory assumed in the author's article \cite{willie_aqc} on computing the braid representations arising from anyon systems. However, the reader need not grasp every detail in order to start doing braid computations using {\sc SageMath}.}

Ribbon categories are certain abelian monoidal, or \textit{tensor}, categories. Tensor categories play an important role in many areas of modern mathematics, such as algebraic geometry, algebraic topology, number theory, mathematical physics, and quantum computation. There are spectacular connections between the theory of tensor categories and representation theory, quantum groups, infinite dimensional Lie algebras, conformal field theory and vertex algebras, operator algebras, and invariants of knots and 3-manifolds which arose from the works of Drinfeld, Moore and Seiberg, Kazhdan and Lusztig, Jones, Witten, Reshetikhin and Turaev, and many others \cite{lusztig_1988,dri_89b,Dri_QHA,moore_seiberg_1989,reshetikhin_turaev_1990,gelfand_kazhdan_1992,kazhdan_lusztig_1993,jones_2006,francesco_mathieu_senechal_2011}. See also \cite{etingof_gelaki_nikshych_ostrik_2017} and references therein. The theory of anyonic computation relies on these connections.

In this context it is helpful to think of category theory as a ``categorification'' of ordinary algebra. In other words, there is a dictionary between the two subjects, such that the usual algebraic structures can be recovered from the categorical structures by passing to the isomorphism classes of objects \cite{etingof_gelaki_nikshych_ostrik_2017}. For instance, the notion of a category is the categorification of the notion of a set\footnote{There is a foundational subtlety that is not relevant in our context: in general objects and morphisms in a category are defined to form a \textit{class} rather than a set. This is to circumvent Russell's paradox regarding the existence of a set of all sets. In this work we deal only with \textit{small} categories, which means their objects and hom-sets are indeed sets.} and the notion of a monoidal category is the categorification of the notion of a monoid. 

This correspondence allows for a theory of algebra representations without vectors, in the same sense that category theory allows for a theory of sets without elements. For instance, the representation theory of a quasitriangular (braided) Hopf algebra may be formulated in terms of a ribbon category.

In the setting of quantum computing, the correspondence allows us to capture the dynamics of anyons, modeled as simple objects in an RFC, without knowing much about what they ``actually look like.'' This is explained in \cite{willie_aqc}.

This \chorpaper\ is structured as follows. \Cref{rfc framework} introduces ribbon categories and necessary terminology. In addition, it discusses the induced braid representations and briefly presents the associated diagrammatical calculus, which is ubiquitous and useful for computations.\paper{ \Cref{not and conv} fixes our notation for quantum groups.} Then \Cref{rfc from q gps} elucidates the connection between semisimple ribbon categories and the representation theory of quantum groups. The aim is to show that the category $\Uqgmodcat$ of finite-dimensional modules over the quantum group $\Uqg$ is a ribbon category whose braiding is generated infinitesimally. This means the braid group representation induced by the categorical braiding of $\Uqgmodcat$ on any tensor product $V^{\otimes m}$ is equivalent to the representation induced by the exponential of the symmetric invariant $2$-tensor on $\Ug$ defined by the enveloping algebra's Casimir element. This representation has a geometric interpretation in terms of the monodromy of the Knizhnik-Zamolodchikov equations, which motivates the connection between certain conformal field theories and the RFC framework \cite{Dri_QHA_KZ_eqns}. Demonstrating this equivalence requires showing that the braiding on $\Uqgmodcat$ is induced by the $R$-matrix of a related topological ribbon algebra $\widetilde{\Uqg}$, which can be understood as a certain completion of $\Uqg$.

\section{The ribbon category framework}
\label{rfc framework}

We assume the notion of a category, subcategory, functor, and their equivalences. Otherwise we start from the beginning. MacLane \cite{maclane_1988} provides a great introduction to the subject of category theory. 

	\subsection{Abelian categories}
	To begin, we need the notion of \textit{additive} categories.

\begin{dfn}\cite[Definition~1.2.1]{etingof_gelaki_nikshych_ostrik_2017}
	An \textit{additive category} is a category $\ccat$ satisfying the following axioms:
	\begin{enumerate}[(i)]
		\item Every hom-set $\Hom(X, Y)$ is equipped with the structure of an abelian group (written additively) making composition of morphisms biadditive.
		\item There exists a zero object $0 \in \ccat$ such that $\Hom(0, 0) = 0$.
		\item (Direct sums exist.) For any $X, Y \in \ccat$ there exists an object $S \in \ccat$ and morphisms $p_X \colon S \to X$, $p_Y \colon S \to Y$, $i_X\colon X \to S$, and $i_Y \colon Y \to S$ such that $p_X i_X = \id_X$, $p_Y i_Y = \id_Y$, and $i_X p_X + i_Y p_Y = \id_S$.
	\end{enumerate}
\end{dfn}

A universal property shows the object $S$ in (iii) is unique up to isomorphism. It is the \textit{direct sum} of $X$ and $Y$, and we denote it by $X \oplus Y$. This means every additive category $\ccat$ is equipped with a bifunctor $\oplus\colon \ccat \times \ccat \to \ccat$.

The categories considered in this work are \textit{$\mathbb{C}$-linear}, or \textit{enriched over} $\mathbb{C}$, which means every hom-set $\Hom(X, Y)$ is equipped with the structure of a complex vector space making composition of morphisms $\mathbb{C}$-linear.

Introducing the notion of a \textit{decomposition} of an object requires the notion of \textit{kernels} and \textit{cokernels}. These are categorifications of the usual notions of kernels and cokernels in linear algebra. Let $\ccat$ be an additive category and let $f \colon X \to Y$ be a morphism in $\ccat$. The \textit{kernel} $\ker(f)$ of $f$ (if it exists) is an object $K$ together with a morphism $k\colon K \to X$ such that $fk = 0$. If $k'\colon K' \to X$ is such that $fk' = 0$ then there exists a unique morphism $\ell\colon K' \to K$ such that $k \ell = k'$. Thus when $\ker(f)$ exists it is unique up to a unique isomorphism.

Dually, the \textit{cokernel} $\coker{f}$ of $f$ (if it exists) is an object $C$ together with a morphism $c\colon Y \to C$ such that $cf = 0$. Similarly, if $c'\colon Y \to C'$ is such that $c'f = 0$ then there exists a unique morphism $\ell\colon C \to C'$ such that $\ell c = 0$. Therefore if $\coker{f}$ exists then it is unique up to a unique isomorphism.

\begin{dfn}\cite[Definition~1.3.1]{etingof_gelaki_nikshych_ostrik_2017}
	An \textit{abelian category} is an additive category $\ccat$ in which for every morphism $f \colon X \to Y$ there exists a sequence
	$$
		K \xrightarrow{\,\smash{{\ensuremath{\scriptstyle k}}}\,}
		X \xrightarrow{\,\smash{{\ensuremath{\scriptstyle i}}}\,}
		I \xrightarrow{\,\smash{{\ensuremath{\scriptstyle j}}}\,}
		Y \xrightarrow{\,\smash{{\ensuremath{\scriptstyle c}}}\,}
		C
	$$
	with the following properties:
	\begin{enumerate}[(i)]
		\item $ji = f$,
		\item $(K, k) = \ker(f)$ and $(C, c) = \coker{f}$,
		\item $(I, i) = \coker{k}$ and $(I, j) = \ker(c)$.
	\end{enumerate}
	The object $I$ is the \textit{image} of $f$ and we denote it by $\mathrm{Im}(f)$.
\end{dfn}

The notion of a kernel allows us to identify subobjects in a given object $X$. Concretely, a morphism $f\colon X \to Y$ in an abelian category is a \textit{monomorphism} if $\ker(f) = 0$. Then a \textit{subobject} of $X$ is an object $Y$ together with a monomorphism $i\colon Y \to X$. In \ddichotomy{\Cref{anyonic qc ch}}{the quantum computing application}, we interpret subobjects of a tensor product as the possible outcomes of anyon fusion.

The following theorem justifies thinking of direct sums, morphisms, kernels and cokernels in an abelian category in terms of ordinary linear algebra.

\begin{thm}\cite{freyd_1966}
	Every abelian $\mathbb{C}$-linear category is equivalent, as an additive category, to a full subcategory of the category of modules over a unital associative $\mathbb{C}$-linear algebra.
\end{thm}

Moreover, a $\mathbb{C}$-linear abelian category $\ccat$ is \textit{finite} if it is equivalent to the category $\Amodcat$ of finite-dimensional modules over a finite-dimensional $\mathbb{C}$-algebra $A$ \cite[Definition~1.8.5]{etingof_gelaki_nikshych_ostrik_2017}. Note that $\ccat$ does not uniquely or canonically determine the algebra $A$, but rather its Morita equivalence class.

Crucially, the notion of a kernel allows us to define a distinguished set of elements in an abelian category. A non-zero object $X$ in an abelian category $\ccat$ is \textit{simple} if its only subobjects are $0$ and $X$. In a $\mathbb{C}$-linear abelian category, $X$ is simple if and only if $\End(X) \cong \mathbb{C}$. If $\ccat$ is \textit{semisimple}, the simple objects can be viewed as a generating set for $\mathrm{Ob}(\ccat)$ under $\oplus$.

\begin{dfn}
	An abelian category $\ccat$ is \textit{semisimple} if each object is semisimple; an object in $\ccat$ is \textit{semisimple} if it is a finite direct sum of simple objects. 
\end{dfn}

Since we only consider \textit{finite} direct sums, there is an alternative characterization of finite $\mathbb{C}$-linear abelian categories: a $\mathbb{C}$-linear abelian category is \textit{finite} if and only if it has finitely many simple objects and its hom-sets are finite-dimensional.

	\subsection{Tensor categories}
	Whereas abelian categories are equipped with an additive structure, \textit{monoidal categories} are equipped with a multiplicative one.

\begin{dfn}\label[defn]{monoidal category}\cite[Section~VII.1]{maclane_1988}
	A \textit{monoidal category} $(\ccat, \otimes, I, a, l, r)$ is a category $\ccat$ with a bifunctor $\otimes: \ccat \times \ccat \to \ccat$, a unit object $I$, and natural isomorphisms $a, l, r$ satisfying the following conditions. First, the following diagram commutes for all triples $(U, V, W)$ of objects in $\ccat$.
\begin{equation}\label{pentagon}
\scalebox{0.5}{
\begin{tikzpicture}[line width=1.1]
	\tikzmath{\apo = 6; \ss=1.7;}
	\node[scale=\ss] at ($(90:\apo)$) {$(U \otimes V) \otimes (W \otimes X)$};
	\node[scale=\ss] at ($(90 + 1*72:\apo)$) {$\big((U \otimes V) \otimes W\big) \otimes X$};
	\node[scale=\ss] at ($(90 + 2*72:\apo)$) {$\big(U \otimes (V \otimes W)\big) \otimes X \quad\quad\quad\quad$};
	\node[scale=\ss] at ($(90 + 3*72:\apo)$) {$\quad\quad\quad\quad U \otimes \big((V \otimes W) \otimes X\big)$};
	\node[scale=\ss] at ($(90 + 4*72:\apo)$) {$U \otimes \big(V \otimes (W \otimes X)\big) $};

	\tikzmath{\start = 0.3; \tar = 1 - \start;}
	\foreach[count=\next] \curr in {0,...,4} {
		\pgfmathsetmacro\ti{90 + int(\curr)*72}
		\pgfmathsetmacro\tf{90 + int(\next)*72}
		\coordinate (vi) at ($(\ti:\apo)!\start!(\tf:\apo)$);
		\coordinate (vf) at ($(\ti:\apo)!\tar!(\tf:\apo)$);
		\pgfmathparse{int(mod(\curr,4))}
		\ifnum\pgfmathresult=0
			\draw[{Latex[length=3mm]}-] (vi) -- (vf);
		\else
			\draw[-{Latex[length=3mm]}] (vi) -- (vf);
		\fi
	}

	\tikzmath{\apo = \apo; \ss = 0.8*\ss; }
	\node[scale=\ss] at ($(90 + 0.5*72 + 6:\apo)$) {$a_{U\otimes V, W, X}$};
	\node[scale=\ss] at ($(90 + 1.5*72:\apo)$) {$a_{U, V, W}\otimes \id_X \quad\quad$};
	\node[scale=\ss] at ($(0, -\apo)$) {$a_{U, V\otimes W, X}$};
	\node[scale=\ss] at ($(90 + 3.5*72:\apo)$) {$\quad\quad \id_U \otimes a_{V, W, X}$};
	\node[scale=\ss] at ($(90 + 4.5*72 - 6:\apo)$) {$a_{U, V, W\otimes X}$};
\end{tikzpicture}}
\end{equation}
The natural isomorphism $a$ is known as the \textit{associativity constraint}. In addition, the associativity constraint is compatible with the \textit{left} and \textit{right unit constraints} $l$ and $r$, which means the following diagram commutes for all $U, V \in \ccat$.
	\begin{equation}\label{diag:triangle}
	\begin{tikzcd}[column sep=0.5cm, row sep=0.5cm]
	(V \otimes I) \otimes W \ar[rr, "a_{V, I, W}"] \ar[dr, "r_V \otimes \id_W"']
	&&
	V \otimes (I \otimes W) \ar[dl, "\id_V \otimes l_{W}"] \\
	&
	V \otimes W
	&
	\end{tikzcd}
	\end{equation}
	A monoidal category is said to be \textit{strict} if the associativity and unit constraints $a, l, r$ are all identities of the category.
\end{dfn}

A \textit{tensor category} is a monoidal abelian category \cite{etingof_gelaki_nikshych_ostrik_2017}. This means every tensor product of objects in a tensor category is a direct sum. In other words, tensor categories are the categorifications of rings.

If $\mathbb{C}$-linear abelian categories are categories of modules over a $\mathbb{C}$-algebra, then $\mathbb{C}$-linear \textit{tensor} categories are categories of modules over a complex \textit{bi}algebra: compatible algebra and coalgebra structures are needed to define module structures on tensor products.

Notice that any bialgebra corresponding to a strict tensor category must be coassociative. Relaxing strictness results in the notion of a \textit{quasi-bialgebra}, which is coassociative up to conjugation by a distinguished element known as an \textit{associator}. \Cref{top_que} discusses quasi-bialgebras in more detail.
	
	\subsection{Braids and ribbons}
	\label{braided and ribbon cat}
	In order to construct representations of the braid group, we need \textit{braided categories}.

\begin{dfn}\label[defn]{braided tensor category}
	A tensor category $(\ccat, \otimes, I, a, l, r)$ is \textit{braided} if there exists a commutativity constraint $c$, called a \textit{braiding}, such that for any pair of objects $V, W$ the map
	\begin{equation*}
	c_{V, W}: V \otimes W \to W \otimes V
	\end{equation*}
	is a natural isomorphism. We require that $c$ satisfies the Hexagon Axioms, encoded in the commutativity following diagrams.
	\begin{equation}\label{hexagon1}
	\begin{tikzcd}[column sep={2cm,between origins}, row sep={2cm,between origins}]
	& U \otimes(V \otimes W) \arrow[rr, "c_{U, V\otimes W}"] && (V \otimes W) \otimes U \arrow[rd, "a_{V, W, U}"] &  \\
	(U \otimes V) \otimes W \arrow[ru, "a_{U, V, W}"]  \arrow[rd, "c_{U,V} \otimes \id_W"'] &  &&  & V \otimes(W \otimes U) \\
	& (V \otimes U) \otimes W \arrow[rr, "a_{V, U, W}"'] && V \otimes (U \otimes W) \arrow[ru, "\id_V \otimes c_{U, W}"'] &
	\end{tikzcd}
	\end{equation}
	\begin{equation}\label{hexagon2}
	\begin{tikzcd}[column sep={2cm,between origins}, row sep={2cm,between origins}]
	& (U \otimes V) \otimes W \arrow[rr, "c_{U \otimes V, W}"] &&  W \otimes (U \otimes V) \arrow[rd, "\inv{a_{W, U, V}}"] &  \\
	U \otimes (V \otimes W) \arrow[ru, "\inv{a_{U, V, W}}"]  \arrow[rd, "\id_U\otimes c_{V,W} "'] &  &&  & (W \otimes U) \otimes V \\
	& U \otimes (W \otimes V) \arrow[rr, "\inv{a_{U, W, V}}"'] && (U \otimes W) \otimes V \arrow[ru, " c_{U, W} \otimes \id_V"'] &
	\end{tikzcd}
	\end{equation}

	The braiding $c$ in a braided category $\ccat$ is \textit{symmetric} if
	\begin{equation}\label{rel:symmetric braiding in category}
		c_{W, V} \circ c_{V, W} = \id_{V \otimes W}
	\end{equation}
	for all objects $V, W$ in $\ccat$. In this case $\ccat$ is a \textit{symmetric category}.
\end{dfn}

Note that in a braided category, the isomorphism $c_{V, V} \colon V \otimes V \to V \otimes V$ is \textit{not} required to be involutive; in general, we obtain non-trivial isomorphisms $V \otimes W \to W \otimes V \to V \otimes W$.

Moreover, observe that braided tensor categories are still categories of modules over a bialgebra, but the bialgebra is not cocommutative in general. However, it is always \textit{quasitriangular (braided)}: cocommutativity is enforced only up to conjugation by a distinguished element called the universal $R$-matrix.

Next we show how to obtain representations of the \textit{Artin braid group} from a braided category. For starters, recall the braid group $B_m$ is generated by $\sigma_1, \ldots, \sigma_{m-1}$ subject to the \textit{braid relations}:
\begin{align}\label{braid rels}
\sigma_{i} \sigma_{i+1} \sigma_{i} = \sigma_{i+1} \sigma_i \sigma_{i+1},
\quad
\sigma_i \sigma_j = \sigma_j \sigma_i,
\quad \text{for} \quad |i-j| > 1.
\end{align}

The braid group $B_m$ models isotopy classes of braids on $m$ strands, so it has a nice diagrammatic interpretation. For instance, \Cref{fig braid generator} illustrates the braid generator $\sigma_2$ in $B_4$. By convention the $(i+1)$st strand in $\sigma_i$ crosses over the $i$th. Vertical concatenation implements multiplication in the group. Like \cite{wang_2010}, we adopt the \textit{optimistic convention}: we parse diagrams from bottom to top, so ``time'' flows upwards.
\begin{figure}[!h]
	\begin{center}
		\begin{tikzpicture}[scale=0.8]
 		\draw[thick] (4,0) to [out=90,in=-90] (2,2);
 		\path[fill=white] (3,1) circle (.2);
 		\draw[thick] (2,0) to [out=90,in=-90] (4,2);
 		\draw[thick] (0,0) to [out=90,in=-90] (0,2);
 		\draw[thick] (6,0) to [out=90,in=-90] (6,2);
 		\draw[fill=gray] (0,0) circle (.15);
 		\draw[fill=gray] (2,0) circle (.15);
 		\draw[fill=gray] (4,0) circle (.15);
 		\draw[fill=gray] (0,2) circle (.15);
 		\draw[fill=gray] (2,2) circle (.15);
 		\draw[fill=gray] (4,2) circle (.15);
 		\draw[fill=gray] (6,0) circle (.15);
 		\draw[fill=gray] (6,2) circle (.15);
 		\node at (3,-1.2) {$\sigma_2$};
 		\node at (0,3.5) {\quad};
		\end{tikzpicture}
	\end{center}
	\caption{A diagrammatic interpretation of the braid generator $\sigma_2 \in B_4$.}
	\label{fig braid generator}
\end{figure}
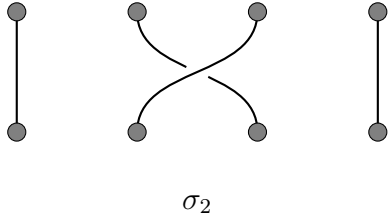

Now let $V$ be any object in a braided category $\ccat$ and set $\check{R} = c_{V, V}$. The map $\check{R}$ is known as a \textit{braiding} operator because it induces braid representations as follows. The Hexagon Axioms imply $\check{R}$ satisfies the celebrated \textit{Yang-Baxter} equation in $\End(V \otimes V \otimes V)$:
\begin{equation}\label{yb}
(\check{R} \otimes \id) (\id \otimes \check{R})	(\check{R} \otimes \id) = (\id \otimes \check{R})(\check{R} \otimes \id) (\id \otimes \check{R}).
\end{equation}
Thus if we let $\check{R}_i$ denote the endomorphism of $V^{\otimes m}$ acting as $\check{R}$ on the $(i, i+1)$st pair of tensor factors, \Cref{yb} implies $\check{R}_i \check{R}_{i+1} \check{R}_i = \check{R}_{i+1} \check{R}_i \check{R}_{i+1} $. Together with the obvious far commutativity $\check{R}_i \check{R}_j = \check{R}_j \check{R}_i$ for $|i - j| > 1$, this means there is a homomorphism $B_m \to \End_{\ccat}(V^{\otimes m})$ satisfying
\begin{equation}\label{braid repn induced by braiding}
	\sigma_i \to \check{R}_i.
\end{equation}

So far our dictionary between categories and algebras does not include Hopf algebras. This requires a notion of dual objects at the category level because at the algebra level, the Hopf algebra antipode turns the dual of any module into a module. Concretely, recall that if $H$ is a Hopf algebra with antipode $S$ and $V$ is an $H$-module, then $H$ acts on $V^*$ via $S$:
\[
	(h \rhd f)(v) = f\left(S(h) v\right),
	\qquad h \in H, \,\, f \in V^*, \,\, v \in V.
\]

Thus we introduce \textit{left} and \textit{right dualities} on a tensor category.

\begin{dfn}\cite[Definition~XIV.2.1]{Kassel}\label[defn]{left duality}
	Let $(\ccat, \otimes, I)$ be a strict tensor category with unit object $I$. It is a tensor category with \textit{left duality} if for each object $V$ of $\ccat$ there exists an object $V^*$ and morphisms
	$$b_V: I \to V \otimes V^* \quad \text{and} \quad d_V: V^* \otimes V \to I$$ satisfying the \textit{rigidity axioms}
	\begin{equation}\label{left rigidity}
		(\id_V \otimes d_V)(b_V \otimes \id_V) = \id_V, \quad \text{ and } \quad (d_{V} \otimes \id_{V^*})(\id_{V^*} \otimes b_V) = \id_{V^*}.
	\end{equation}
	For any morphism $f: V\to W$, we define the \textit{left transpose} $f^*: W^* \to V^*$ of $f$ by the formula
	\begin{align}\label[rel]{defn:transpose of f}
		f^* = (d_W \otimes \id_{V^*})(\id_{W^*} \otimes f \otimes \id_{V^*})(\id_{W^*} \otimes b_V).
	\end{align}

	Alternatively, the tensor category $\ccat$ has \textit{right duality} if for each object $V$ of $\ccat$ there exists an object $\prescript{*}{}{V}$ and morphisms
	$$b'_V\colon I \to \prescript{*}{}{V} \otimes V \quad \text{and} \quad d'_V\colon V \otimes \prescript{*}{}{V} \to I$$
	satisfying
\begin{align}\label{right rigidity}
	(d'_V \otimes \id_V)(\id_V \otimes b'_V) = \id_V
	  \quad \text{and} \quad
	  (\id_{\prescript{*}{}{V}} \otimes d'_V)(b'_V \otimes \id_{\prescript{*}{}{V}})
	   = \id_{\prescript{*}{}{V}}
\end{align}
\end{dfn}

Of course we may define the \textit{right transpose} of a morphism, but we do not need it here.

Although a right duality enjoys properties analogous to those of a left duality, they are different in general. However, when $\ccat$ is a \textit{ribbon category} the left and right duals coincide \cite[Section~XIV.5]{Kassel}.

\begin{dfn}\label[defn]{ribbon cat}
	A strict braided tensor category $\ccat$ with left duality is a \textit{ribbon category} if there exists a family $\theta_V: V \to V$ of natural isomorphisms, called \textit{twists}, such that
	\begin{align*}
		\theta_{V\otimes W} = (\theta_V \otimes \theta_W) c_{V,W}c_{W, V} \quad \text{and} \quad \theta_{V^*} = (\theta_V)^*.
	\end{align*}
 	\Cref{defn:transpose of f} defines the transpose $f^*$ of a morphism $f$.
\end{dfn}

Moreover, a ribbon structure also controls the categorical braiding: while the symmetry constraint $c_{W, V} \circ c_{V, W} = \id_{V \otimes W}$ is too strict in general, in a ribbon category we require this composition to be a kind of coboundary instead \cite{kassel_rosso_turaev_1997}. This has important implications for the braiding in ribbon categories, especially for semisimple ones.

In particular, the twists determine induced action of braid generators, up to a choice of signs, in any semisimple ribbon category. Concretely, suppose $\ccat$ is a semisimple ribbon category and let $(\theta_W)_{W \in \ccat}$ denote a family of twists in $\ccat$. For each simple object $W$, let $\check{\theta}_W$ denote the complex scalar defined by $\theta_W = \check{\theta}_W \id_W \in \End(W) \cong \mathbb{C}$. Now choose any simple object $V$. Since $\ccat$ is semisimple, there is a set $\mathcal{I}$ of simple objects in $\ccat$ and integers $N_X^{V, V}$ such that
	\begin{equation*}
		V \otimes W = \bigoplus_{X \in \mathcal{I}} N^{V, V}_X X.
	\end{equation*}
	Suppose each $N_X^{V, V} = 1$. In this case, we say the decomposition of $V \otimes V$ is \textit{multiplicity-free}. This means there are maps $i_X \colon X \to V \otimes V$ and $p_X \colon V \otimes V \to X$ for every $X \in \mathcal{I}$ satisfying
	$$
		i_X p_Y = \delta_{X, Y} \id_X
		\quad \text{and} \quad
		\sum_{X \in \mathcal{I}} i_X p_X = \id_{V \otimes V}.
	$$
	Let $P[X] \colon V \otimes V \to V \otimes V$, with $P[X] = i_X p_X$, denote the projection of $V \otimes V$ onto its subobject $X$. The set of $P[X]$, for $X \in \mathcal{I}$ is a complete set of orthogonal idempotents: 
\[
	P[X] P[Y] = \delta_{X, Y} \id_{V \otimes V}
	\quad \text{and} \quad
	\sum_{X \in \mathcal{I}} P[X] = \id_{V \otimes V}.
\]
Finally, let $\check{R} \coloneqq c_{V, V}$ denote the braiding operator on $V \otimes V$ and set $\check{R}^{V, V}_X = \sqrt{\check{\theta}_X \check{\theta}_V^{-2}} \in \mathbb{C}$ for every $X \in \mathcal{I}$.

\begin{thm}\label{semisimple braiding ribbon cat}
	With the notation of the previous paragraph in mind,
	there exists a sign $\epsilon\colon \mathcal{I} \to \{\pm 1\}$ such that 
	\begin{equation*}
		\check{R} = \sum_{X \in \mathcal{I}} \epsilon(X) \check{R}^{V, V}_X P[X].
	\end{equation*}
\end{thm}
\begin{proof}
	The naturality of the twist implies $p_X \theta_{V \otimes V} = \theta_X p_X$ for any $X \in \mathcal{I}$. By definition $\theta_{V \otimes V} = (\theta_V \otimes \theta_V) \check{R}^2$ and $\theta_V \otimes \theta_V = \check{\theta}_V \check{\theta}_V$ because $V$ is simple. Substituting and rearranging yields $p_X \check{R}^2 = \check{\theta}_X \check{\theta}_V^{-2} p_X$. Pre-multiplying by $i_X$ and adding over $X \in \mathcal{I}$ implies $\check{R}^2 = \sum_X \left(\check{R}^{V, V}_X\right)^2 P[X]$. The theorem follows because the $P[X]$ form a complete set of orthogonal idempotents.
\end{proof}

Thus from now on we focus on semisimple ribbon categories, which induce braid representations characterized by their twists, as in \Cref{semisimple braiding ribbon cat}. In particular, \cite{willie_aqc} considers \textit{ribbon fusion categories} exclusively.

\begin{dfn}\label[defn]{rfc defn}
	A \textit{ribbon fusion category (RFC)} is a finite $\mathbb{C}$-linear semisimple ribbon category. 
\end{dfn}
\begin{rmk}
	In the literature, RFCs are sometimes known as \textit{premodular categories} \cite{bruguieres}. An RFC becomes a \textit{modular category} if its braiding satisfies a non-degeneracy condition making the so-called \textit{$S$-matrix} non-singular \cite[Section~II.1.4]{turaev_2016}.
\end{rmk}

	\subsection{Diagrammatical calculus}
	


\section{Conventions for Lie algebras and quantum groups}
\label{not and conv}
In this section we fix our notation for Lie algebras and recall some basic definitions. We will use $e_i$ throughout to denote the standard basis of $\mathbb{R}^n$.\longer{Let $\lieg$ denote either the reductive complex Lie algebra $\mathfrak{gl}_n$, or a finite-dimensional complex semisimple Lie algebra. We will mainly consider the cases $\lieg = \mathfrak{gl}_n, \mathfrak{so}_n$, and $\mathfrak{sp}_n$.} 

\longer{\paragraph{Lie algebra presentation}}
Following \cite{chari_pressley_1994}, we use the \textit{Chevalley presentation}. Let $\lieg$ denote a complex semisimple Lie algebra, let $A = [a_{ij}]$ denote its \textit{generalized Cartan matrix}, and let $D = \mathrm{diag}(d_1, \ldots, d_r)$ denote the diagonal matrix of root lengths such that $DA$ is symmetric positive definite. Recall $a_{ii} = 2$ and $a_{ij} \leq 0$ whenever $i\neq j$.
	
The semisimple algebra $\lieg$ is generated by $H_i, E_i$, and $F_i$, for $i = 1, \ldots, r$, subject to the relations
\begin{equation}\label{chevalley presentation of g}
	\begin{gathered}
	[H_i, H_j] = 0, \quad [H_i, E_j] = a_{ij} E_j, \quad [H_i, F_j] = -a_{ij} F_j, \quad [E_i, F_j] = \delta_{ij} H_i,  \\
	\end{gathered}
\end{equation}
along with the \textit{Serre relations} for $i \neq j$:
\begin{equation}\label{Serre rels of g}
	(\mathrm{ad}_{E_i})^{1-a_{ij}}(E_j) = 0, \quad \text{and} \quad (\mathrm{ad}_{F_i})^{1-a_{ij}}(F_j) = 0
\end{equation}
\begin{remark}
	Authors like Jantzen \cite{Jantzen96}, Lusztig \cite{lusztig_1988}, and Sawin \cite{sawin_1996} prefer the notation $X_i^+ = E_i$ and $X_i^- = F_i$\longer{, suggestive of positive and negative roots in general, and of ``raising'' and ``lowering'' operators in this context}.
\end{remark}

We will denote by $\Ug$ the \textit{universal enveloping algebra} associated to $\lieg$\longer{ \cite[Chapter~10]{Bump}. This associative unital algebra has the same generators as $\lieg$, subject to relations in \eqref{chevalley presentation of g}, where in this case $[A, B] = AB - BA$ denotes the usual algebra commutator. The generators are also subject to the Serre relations \eqref{Serre rels of g}, which are written in this context as
\begin{equation*}\label{classical Serre relations}
	\sum_{k=0}^{1-a_{ij}}(-1)^k \binom{1-a_{ij}}{k} (X_i)^k X_j (X_i)^{1- a_{ij}-k} = 0.
\end{equation*}
In the last equality $X$ is either an $E$ or an $F$}.

Now fix the \textit{Cartan subalgebra} $\mathfrak{h}$ generated by the $H_i$\longer{ and let $\mathfrak{b}$ denote the \textit{Borel subalgebra} containing $\mathfrak{h}$}. The \textit{simple roots} of $\lieg$ are the linear functionals $\alpha_i\colon \mathfrak{h} \to \mathbb{C}$ satisfying
\begin{equation*}
	\alpha_i(H_j) = a_{ji}.
\end{equation*}
\longer{We use $\varPhi$ to denote the set of simple roots and we set 
\begin{equation*}
	\mathcal{Q} = \bigoplus_{\alpha \in \varPhi} \mathbb{Z} \alpha 
	\quad \text{and} \quad 
	\mathcal{Q}^+ = \bigoplus_{\alpha \in \varPhi} \mathbb{Z}_+ \alpha.
\end{equation*}}
We let
\begin{equation}\label{weights and dominant weights}
	\mathcal{P} = \{ \lambda \in \mathfrak{h}^* \mid \lambda(H_i) \in \mathbb{Z} \} \quad \text{and} \quad \mathcal{P}^+ = \{ \lambda \in \mathfrak{h}^* \mid \lambda(H_i) \in \mathbb{Z}_+ \}
\end{equation}
denote the lattice of \textit{weights} and \textit{dominant weights}, respectively.

\longer{\paragraph{Bilinear form}}
When $\lieg$ is semisimple, there exists a unique non-degenerate symmetric invariant bilinear form $\langle , \rangle\colon \lieg \times \lieg \to \mathbb{C}$ such that
\begin{gather*}
\begin{split}
	\langle H_i, H_j \rangle = \inv{d_j} a_{ij}, \quad \langle H_i, E_j \rangle = \langle H_i, F_j \rangle = 0, \\
	\langle E_i, E_j \rangle = \langle F_i, F_j \rangle =  0,  \quad \text{and} \quad \langle E_i, F_j \rangle = \inv{d_i} \delta_{ij},
\end{split}
\end{gather*}
for all $i, j$ \cite[Theorem~2.2]{kac_1985}. When $\lieg = \gln$ we take $\langle , \rangle$ to be the non-degenerate trace bilinear form of the natural representation.
	
The bilinear form $\langle , \rangle$ induces an isomorphism of vector spaces $\nu\colon \mathfrak{h} \to \mathfrak{h}^*$ under which $H_i$ corresponds to the \textit{coroot} $\alpha_i^\vee = \inv{d_i} \alpha_i$. The induced form on $\mathfrak{h}^*$ satisfies
\begin{equation}\label{inner prod on roots}
	\langle \alpha_i, \alpha_j \rangle = d_i a_{ij}.
\end{equation}
Since $a_{ii} = 2$, we must have $d_i = \langle \alpha_i, \alpha_i \rangle / 2$. Thus we obtain a formula for the \textit{Cartan integers}:
\begin{equation}\label{cartan ints as inn prod}
	a_{ij} = \langle \alpha_i^\vee, \alpha_j \rangle 
	= { 
		2 \langle \alpha_i^\vee, \alpha_j \rangle 
		\over 
		\langle \alpha_i, \alpha_i \rangle
	  }.
\end{equation} 
Notice the bilinear form induced by $\nu$ is normalized so that 
$$\langle \alpha_i, \alpha_i \rangle = 2d_i.$$
When $\lieg$ is simple\longer{, the Cartan matrix $A$ has entries $a_{ij} \in \{-3, -2, -1, 0, 2\}$, so without loss of generality we may assume that $D$ has entries $d_i \in \{1, 2, 3\}$, which implies} $\langle \alpha_i, \alpha_i \rangle = 2$ for \textit{short} roots.

For concreteness and convenience, we record some relevant Cartan matrices: 
\begin{align}\label{cartan mats}
\begin{gathered}
	A_n = 
	\begin{bmatrix*}[r]
	2  & -1     &        &        & \\
	-1 & 2      & -1     &        & \\
	   & \ddots & \ddots & \ddots & \\
	   &        & -1     & 2      & -1 \\
	   &        &        & -1     & 2
	\end{bmatrix*}, 
	\\ \\
	B_n = 
	\begin{bmatrix*}[r]
	\begin{matrix}
	& & &&\\
	& A_{n-1} &&& \\
	& & &&
	\end{matrix} \rvline
	& 
	\begin{matrix*}[c]
	0 \\ \vdots \\ 0 \\ -1
	\end{matrix*} \\
	\cmidrule(lr){1-1}
	\begin{matrix*}[l]
	0 & \cdots & 0 & -2
	\end{matrix*}
	& 2
	\end{bmatrix*}, 
	\quad \text{and} \quad 
	D_n = 
	\begin{bmatrix*}[r]
	\begin{matrix}
	& & &&&\\
	& A_{n-1} &&&& \\
	& & &&&
	\end{matrix} \rvline
	& 
	\begin{matrix*}[c]
	0 \\ \vdots \\ 0 \\ -1 \\ 0
	\end{matrix*} \\
	\cmidrule(lr){1-1}
	\begin{matrix*}[l]
	0 & \cdots & 0 & -1 & 0
	\end{matrix*}
	& 2
	\end{bmatrix*}.
\end{gathered}
\end{align}
We label the matrices by the Lie type of the root system to which they are associated. The corresponding diagonal   root lengths matrices are described by $d = (1, \ldots, 1)$ for the root systems of types $A_n, D_n$, and by $d = (2, \ldots, 2, 1)$ for type $B_n$.

If $\alpha \in \mathcal{Q}$, define the \textit{root space} 
$$\lieg_\alpha = \{ x \in \lieg \mid [h, x] = \alpha(h) x \text{ for all } h \in \mathfrak{h}\}.$$
The elements of $\Delta = \{ \alpha \in \mathcal{Q} \mid \alpha \neq 0, \lieg_\alpha \neq 0 \}$ are the \textit{roots} of $\lieg$. The intersection $\Delta^+ = \Delta \cap \mathcal{Q}^+$ denotes the set of \textit{positive roots}. The \textit{negative roots} are given by $\Delta^- = - \Delta^+$ and $\Delta = \Delta^+ \coprod \Delta^-$ \cite[Chapter 1]{kac_1985}. Distinct root spaces are orthogonal with respect to the bilinear form $\langle , \rangle$:
$$\langle \lieg_\alpha, \lieg_\beta \rangle \text{ if } \alpha \neq \beta \text{ and }  \langle \lieg_\alpha, \mathfrak{h} \rangle \text{ if } \alpha \neq 0$$
\cite[Theorem 2.2(c)]{kac_1985}. The Lie algebra $\lieg$ has a triangular decomposition with respect to its roots. As vector spaces, 
\begin{equation*}\label{triangular decomp g}
	\lieg = \big(\bigoplus_{\alpha < 0} \lieg_\alpha \big) \oplus \mathfrak{h} \oplus \big(\bigoplus_{\alpha > 0} \lieg_\alpha \big).
\end{equation*}

We let $\rho$ denote the half sum of the positive roots in $\mathfrak{h}^*$. Alternatively, the \textit{Weyl vector} $\rho$ is uniquely characterized by $\langle \alpha_i^\vee, \rho \rangle = 1$. Thus its image $\rho^* \in \mathfrak{h}$ under $\inv{\nu}$ satisfies $\alpha_i(\rho^*) = d_i$. Concretely, if $b_i = \sum_j (\inv{A})_{ji} d_j$, then
\begin{equation}\label{weyl vector}
	\rho = \sum_{i} b_i \alpha_i^\vee \quad \text{and} \quad \rho^* = \sum_{i} b_i H_i.
\end{equation}

The \textit{fundamental weights} $\omega_i \in \mathfrak{h}^*$  are defined by $\omega_i(H_j) = \delta_{ij}$, or equivalently, $\langle \omega_i, \alpha_j^\vee \rangle = \delta_{ij}$. The $\omega_i$ form a basis of $\mathcal{P}$ and clearly
$$\alpha_i = \sum_j a_{ji} \omega_j.$$ 
Since $C = DA$ is symmetric $\alpha_i = \sum_j c_{ij} (\inv{d_j} \omega_j)$ and therefore
$$\omega_i = \sum_j (D\inv{A})_{ij} \alpha_j^\vee$$
because $C$ is invertible. Then
$\langle \omega_i, \omega_j \rangle = d_i (\inv{A})_{ij}$ follows from the definition of $\omega_i$. 

Now consider any weights $\lambda = \sum_i \lambda_i \omega_i$ and $\mu = \sum_i \mu_i \omega_i$, and notice  $\lambda(H_i) = \lambda_i$ and similarly $\mu(H_i) = \mu_i$. Therefore
\begin{equation}\label{scalar product of weights}
	\langle \lambda, \mu \rangle = \sum_{ij} \lambda_i \mu_j \langle \omega_i, \omega_j \rangle = \sum_{ij} (D\inv{A})_{ij} \lambda(H_i) \mu(H_j).
\end{equation}

\longer{\paragraph{Casimir element}}
Let $X_\alpha$ denote any basis of $\lieg$, and let $X^\alpha$ denote the corresponding dual basis with respect to $\langle \cdot, \cdot \rangle$. The next formula uniquely characterizes the \textit{Casimir element}
\begin{equation}\label{casimir elt}
C = \sum_{\alpha} X_{\alpha}  X^{\alpha}  = \sum_{\alpha} X^{\alpha}  X_{\alpha} \in \Ug.
\end{equation}
The Casimir element is in fact canonical and central \cite[Theorem~10.2]{Bump}. In addition, $C$ acts as the scalar 
\begin{equation}\label{casimir eigenvalues}
	\chi_{\lambda}(C) = \langle \lambda, \lambda + 2\rho\rangle
\end{equation}
on any irreducible $\Ug$-module with highest weight $\lambda$ \cite[Corollary~2.6]{kac_1985}.

\longer{\paragraph{Quantum groups and their representations}}
In what follows, the term ``quantum group'' refers to an associative Hopf algebra $\Uqg$ as presented in \cite[Definition~9.1.1]{chari_pressley_1994}. For convenience, we reproduce the definition. 

Let $\lieg$ denote a semisimple Lie algebra and let $A$ denote its generalized Cartan matrix with the diagonal matrix $D = \mathrm{diag}(d_1, \ldots, d_r)$ such that $DA$ is symmetric. Let $q$ be \textit{indeterminate} and write $q_i = q^{d_i}$.
	The \textit{quantum group} $\Uqg$ is the associative unital algebra over $\mathbb{C}(q)$ generated by the elements $E_i, F_i, K_i$, and $\inv{K_i}$, for $i = 1, \ldots, r$, subject to the relations:
\begin{equation}\label{uq rels}
\begin{gathered}
K_i K_j = K_j K_i, \quad K_i \inv{K_i} = \inv{K_i} K_i = 1, \\
K_i E_j \inv{K_i} = q_i^{a_{ij}} E_j, \quad K_i 
F_j \inv{K_i} = q_i^{-a_{ij}} F_j,  \\
E_i F_j - F_j E_i = \delta_{ij} \frac{K_i - \inv{K_i}}{q_i - \inv{q_i}},
\end{gathered}	
\end{equation}
together with the $q$-Serre relations for $i \neq j$,
\begin{equation}\label{uq Serre rels}
	\sum_{k = 0}^{1-a_{ij}} (-1)^k 
	\begin{bmatrix}
	1 - a_{ij} \\ k
	\end{bmatrix}_{q_i} 
	X_i^k X_j X_i^{1-a_{ij}-k} = 0.
\end{equation}
In the last equality every $X$ is either an $E$ or an $F$. The $q$-integer is given by $[n]_{q_i} = \frac{q_i^n - q_i^{-n}}{q_i - \inv{q_i}}$, the $q$-factorial $[n]_{q_i}! = [n]_{q_i}\cdots [1]_{q_i}$, and the $q$-binomial coefficient is defined analogously. 

Unless otherwise stated, we assume $\Uqg$ is equipped with the co-algebra structure defined by the comultiplication map $\Delta\colon \Uqg \to \Uqg^{\otimes 2}$ satisfying
\begin{align}\label{delta convention}
	\begin{split}
		\Delta(E_i) &= E_i \otimes K_i + 1 \otimes E_i \\
		\Delta(F_i) &= F_i \otimes 1 + \kinv \otimes F_i, 
		\quad \text{and} \\
		\Delta(K_i) &= K_i \otimes K_i,
	\end{split}
\end{align}
as in \cite{lzz_2010}. 

In this work, we only consider finite-dimensional type $(1, \ldots, 1)$ modules for the quantum groups $\Uqg$. Refer to Chapters~9 and~10 in \cite{chari_pressley_1994} for definitions. Every type $(1, \ldots, 1)$ finite-dimensional $\Uqg$-module is semisimple \cite[Theorem~5.17]{Jantzen96} and each simple $\Uqg$-module is in one-to-one correspondence with a $\lieg$-module parametrized by a dominant weight of $\lieg$ \cite[Theorem~5.10]{Jantzen96}. Moreover, corresponding modules have the same weight multiplicities \cite[Lemma~5.14]{Jantzen96}.

\section{Ribbon categories from quantum groups}
\label{rfc from q gps}
\longer{Recall \ref{not and conv}. Let $\lieg$ denote a finite-dimensional complex semisimple Lie algebra or $\gln$.} 

In this section we show that the category $\Uqgmodcat$ of finite-dimensional modules over the Drinfeld-Jimbo quantum group $\Uqg$ is a semisimple ribbon category. The interested reader may refer to \cite{ap_95} to construct RFCs as finite quotients of $\Uqgmodcat$.

In addition, we show that the braiding on $\Uqgmodcat$ is characterized by an infinitesimal braiding on the category of finite-dimensional $\Ug$-modules, which is itself characterized by the infinitesimal braiding on $\Ug$ induced its Casimir element. \Cref{inf_braids} recalls infinitesimal braidings. This characterization of the braiding on $\Uqgmodcat$ has a geometric interpretation in terms of the Knizhnik-Zamolodchikov (KZ) equations; the connection was developed in the works \cite{Dri_QHA_KZ_eqns,kazhdan_lusztig_1993,Kohno} by Drinfeld, Kazhdan and Lusztig, and Kohno, amongst others.

A key step in our construction is showing that the braiding on $\Uqgmodcat$ is generated by the universal $R$-matrix of the topological ribbon algebra $\widetilde{\Uqg}$, defined in \Cref{uqg_modcat} as a certain completion of $\Uqg$. This requires an understanding of topological quantized enveloping algebras, which is developed in \Cref{top_que}. In particular, \Cref{top_que} culminates in a Drinfeld-Kohno theorem: the braiding on \textit{any} topological quantized enveloping algebra over $\lieg$ is generated by some infinitesimal braiding on $\Ug$. 

We omit most proofs where precise references are provided.

	\subsection{Infinitesimal braids}
	\label{inf_braids}
	We recall the notion of infinitesimal braids and show $\Ug$ carries infinitesimal braidings. This means the collection of finite-dimensional $\Ug$-modules is an infinitesimal symmetric category in the sense of \Cref{infinitesimal braided category}. 

\begin{dfn}
	The \textit{infinitesimal braid algebra} $\mathfrak{p}_n$ is the unital associative $\mathbb{C}$-algebra generated by $t_{ij}$, $1 \leq i \neq j \leq n$ subject to the following relations, for all pairwise distinct $i,j,k,l$:
	\begin{equation}\label[rel]{rel:inf braid relations}
	t_{ij} = t_{ji}, \qquad [t_{ij}, t_{kl}] = 0, \qquad [t_{ij}, t_{ik} + t_{jk}] = 0.
	\end{equation}
	Here $[\cdot, \cdot]$ denotes the usual algebra commutator.
\end{dfn}

The \textit{infinitesimal braidings} on $\Ug$, that is, the homomorphisms $\mathfrak{p}_n \to \Ug^{\otimes n}$, are determined by structural properties of $\Ug$ as a cocommutative Hopf algebra. Thus we study infinitesimal braidings in this more general context first.

Let $H$ denote a cocommutative Hopf algebra. Every \textit{symmetric invariant $2$-tensor} on $H$ determines an infinitesimal braiding on $H^{\otimes n}$.

\begin{dfn}\label[defn]{symm inv 2-tensor}
Let $\mathrm{Prim}(H) = \{x \in H \mid \Delta(x) = x \otimes 1 + 1 \otimes x \}$ denote the subset of \textit{primitive elements} in $H \otimes H$. A \textit{symmetric invariant 2-tensor} on $H$ is an element $\theta \in \mathrm{Prim}(H) \otimes \mathrm{Prim}(H)$ such that $\theta_{21} = \theta$ and $[\Delta(x), \theta] = 0$ for all $x \in H$.
\end{dfn}

\begin{prop}\cite[Lemma~XIX.3.2]{Kassel}\label[prop]{symm inv tensors satisfy inf braid rel}
	Let $\theta = \sum_a x_a \otimes y_a$ denote a symmetric invariant $2$-tensor on $H$. For all $i \neq j$, let $\theta_{ij}$ denote the element of $H^{\otimes n}$ defined by
	$$
		\theta_{ij} =
		\sum_a x_a^{(1)} \otimes \cdots \otimes x_a^{(n)},
	$$
	where $x_a^{(i)} = x_a$, $y_a^{(i)} = y_a$, and $x_k^{(\ell)} = 1$ otherwise. There is an algebra homomorphism $\mathfrak{p}_n \to H^{\otimes n}$ satisfying $t_{ij} \to \theta_{ij}$.
\end{prop}
\longer{\begin{proof}
	Since $\theta$ is symmetric, $\theta_{ij} = \theta_{ji}$. The remaining relations follow from the invariance of $\theta$. Without loss of generality, we may assume $n = 3$. Since $[\theta, \Delta(x)] = 0$ for every $x \in H$, it follows that
	\begin{align*}
	[\theta_{12}, \theta_{13} + \theta_{23}] &= \sum_{k, \ell} [x_\ell \otimes y_\ell \otimes 1, x_k \otimes 1 \otimes y_k + 1 \otimes x_k \otimes y_k] \\
	&= \sum_k \bigg[\sum_{\ell} x_\ell \otimes y_\ell, x_k \otimes 1 + 1 \otimes x_k \bigg] \otimes y_k \\
	&= \sum_k [\theta, \Delta(x_k)] \otimes y_k \\
	&= 0.
	&\qedhere
	\end{align*}
\end{proof}}

We will also need the dual notion at the level of categories.

\begin{dfn}\cite[Definition~XX.4.1]{Kassel}\label[defn]{infinitesimal braided category}
An \textit{infinitesimal braiding} on a strict monoidal category $\mathcal{S}$ with involutive braiding $\sigma_{V, W}: V \otimes W \to W \otimes V$ is a family of functorial endomorphisms 
$t_{V, W} : V \otimes W \to V \otimes W,$ 
	defined for all pairs $(V, W)$ of objects in $\mathcal{S}$, which statisfies
\begin{equation}\label[rel]{inf braiding rel1}
\sigma_{V, W} \circ t_{V, W} = t_{W, V} \circ \sigma_{V, W},
\end{equation}
	and 
\begin{equation}\label[rel]{inf braiding rel2}
t_{U,V \otimes W} = t_{U, V} \otimes \id_W + \inv{(\sigma_{U, V} \otimes \id_W)} \circ (\id_V \otimes t_{U, W}) \circ (\sigma_{U, V} \otimes \id_W)
\end{equation}
	for all objects $U, V, W$ in $\mathcal{S}$.

An \textit{infinitesimal symmetric category} is a symmetric braided category equipped with an infinitesimal braiding. 
\end{dfn}

Observe that \Cref{inf braiding rel1} implies \Cref{inf braiding rel2} is equivalent to 
\begin{equation}\label[rel]{inf braiding equiv rel}
t_{U \otimes V,  W} = \id_U \otimes t_{V, W}   + \inv{(\id_U \otimes \sigma_{V, W})} \circ (t_{U, W} \otimes \id_V ) \circ(\id_U \otimes \sigma_{V, W}).
\end{equation}

The following proposition partially explains the terminology: it shows \Cref{inf braiding rel2,inf braiding equiv rel} are infinitesimal versions of the Hexagon \Cref{hexagon1,hexagon2} imposed on a braiding in a braided category. 

\begin{prop}\label[prop]{braiding as deformation of inf braiding}
Let $\ccat$ be a braided category in which the morphisms are given as formal series in a parameter $h$. In particular, suppose the braiding $c_{V, W}$ is of the form 
\begin{equation*}
c_{V, W} = \sigma_{V, W}(\id_{V \otimes W} + h t_{V, W} + O(h^2))
\end{equation*}
for some involutive symmetry $\sigma_{V, W}$. Here $O(h^2)$ denotes a sum of terms of degree $2$ or higher in $h$. Then $t_{V, W}$ satisfies \Cref{inf braiding rel2,inf braiding equiv rel}. 
\end{prop}
\begin{proof}
In view of \Cref{inf braiding rel1}, we need only verify \Cref{inf braiding rel2}. Together with $(\sigma_{U, V} \otimes \id_W)(\id_V \otimes \sigma_{U, W})  = \sigma_{U, V \otimes W}$, \Cref{inf braiding rel1} implies
\begin{align*}
(c_{U, V} &\otimes \id_W) (\id_V \otimes c_{U, W}) 
= (\id_V \otimes \sigma_{U, W}) (\sigma_{U, V} \otimes \id_W) \\
&\qquad + h(\id_V \otimes \sigma_{U, W}) \left((\sigma_{U, V}\otimes \id_W)(t_{U, W} \otimes \id_W) + (\id_V \otimes t_{U, W})(\sigma_{U, V} \otimes \id_W)\right) \\
&= \sigma_{U, V \otimes W}(1 + h ((t_{U, W} \otimes \id_W) + \inv{(\sigma_{U, V} \otimes \id_W)}(\id_V \otimes t_{U, W})(\sigma_{U, V} \otimes \id_W)),
\end{align*}
modulo $h^2$. The Hexagon \Cref{hexagon1} implies 
	\begin{equation*}
	c_{U, V\otimes W} = (c_{U, V} \otimes \id_W) (\id_V \otimes c_{U, W}).
	\end{equation*}
	Identifying the linear terms in the last equality proves the claim.
\end{proof}

Since $H$ is cocommutative, the category $\operatorname{H-Mod}$ of modules over $H$ is a symmetric braided category with the flip map $\tau_{V, W}$ taking $v \otimes w \to w \otimes v$ as its symmetry. 

Moreover, there is a one-to-one correspondence between infinitesimal braidings on $\operatorname{H-Mod}$ and symmetric invariant $2$-tensors on $H$. On one hand, each symmetric invariant $2$-tensor on $H$ defines an infinitesimal braiding on $\operatorname{H-Mod}$: \Cref{symm inv tensors satisfy inf braid rel} implies that if $\theta$ is a symmetric invariant $2$-tensor, there is an infinitesimal braiding on $\operatorname{H-Mod}$ satisfying 
\begin{equation}\label{inf braiding from inv symm tensor}
t_{V, W} (v \otimes w) = \theta \rhd v \otimes w
\end{equation}
for every pair of $H$-modules $V, W$. Conversely, every infinitesimal braiding on $\operatorname{H-Mod}$ arises this way: if $t_{V, W}$ denotes an infinitesimal braiding on $\operatorname{H-Mod}$ then $\theta = t_{H, H}(1 \otimes 1)$ is a symmetric invariant $2$-tensor.

We can characterize the infinitesimal braidings on an infinitesimal symmetric category more precisely when there is additional structure in the category, such as a left duality in the sense of \Cref{left duality}. In such categories, every infinitesimal braiding has the form \cite[Section~XX.4]{Kassel}
\begin{equation}\label{tvw if left duality}
	t_{V, W} = \frac{1}{2}\big(C_{V\otimes W} - \id_V\otimes C_W - C_V  \otimes \id_W\big),
\end{equation}
with $C_V: V \to V$ denoting the natural endomorphism defined by
\begin{equation}\label{cv if left duality}
C_V = - \bigg(\id_V \otimes\big( d_V \circ t_{V^*, V}\big)\bigg) \circ (b_V \otimes \id_V).
\end{equation}

Thus since the subcategory $\Hmodcat$ of finite-dimensional $H$-modules has left duality,  every infinitesimal braiding on $\Hmodcat$ is induced by the action of a \textit{single} element $\sum_i x_i y_i \in H$, such that $\theta = \sum_i x_i \otimes y_i$ is a symmetric invariant $2$-tensor.

Specializing $H = \Ug$, this means each infinitesimal braiding on $\Ugmodcat$ is induced by the action of an element $\sum_i x_i y_i$, with $x_i, y_i \in \mathrm{Prim}(\Ug) = \lieg$. An important example concerns the Casimir element $C = \sum_\alpha X_\alpha X^\alpha$ introduced in \Cref{casimir elt}. Note that the associated $2$-tensor
\begin{equation}\label{t via casimir}
	t = {1 \over 2}\left(\Delta(C) - C \otimes 1 - 1 \otimes C\right) 
	= \sum_\alpha X_\alpha \otimes X^\alpha
\end{equation}
is symmetric and it is invariant because $C$ is canonical and central in $\Ug$. 

We shall see in \Cref{uqg_modcat} that this infinitesimal braiding generates the braiding on the category of finite-dimensional $\Uqg$-modules. To see this, however, we will need to go through a certain topological quantized enveloping algebra first. 
	
	\subsection{(Topological) quantized enveloping algebras}
	\label{top_que}
	We define topological quantized enveloping algebras (QUEs) in the sense of \cite{Dri_qgps}. As a certain deformation of $\Ug$, a QUE is equipped with a braiding that is generated infinitesimally by a symmetric invariant $2$-tensor on $\Ug$. The Drinfeld-Kohno \Cref{every que is Agtheta} makes this precise. We single out the QUE $\Uh$ because it carries a braiding whose action on any tensor product is equivalent to that induced by the matrix $e^{ht}$, with $t$ as in \Cref{t via casimir}, and, in addition, it has a ribbon element.

We begin by recalling topological algebras and deformations. Fix a parameter $h$ and let $\mathbb{C}[[h]]$ denote the ring of formal power series in $h$ with complex coefficients.

\begin{dfn}\cite[A.4.1]{bonneau_flato_gerstenhaber_pinczon_1994}\label[defn]{topological alg}
	A \textit{topological vector space (t.v.s.)} is a complex vector space with a locally convex Hausdorff topology. A \textit{topological algebra} over a ring $R$ is a triple $(A, \mu, \eta)$ where $A$ is a complete t.v.s., and the multiplication $\mu: A \widetilde{\otimes} A \to A$ and unit $\eta: R \to A$ are continuous linear maps satisfying
	\begin{equation*}
	\mu \circ (\mu \widetilde{\otimes} \id_A) = \mu \circ (\id_A \widetilde{\otimes} \mu)
	\quad \text{and} \quad
	\mu \circ (\eta \widetilde{\otimes} \id_A) = \id_A = \mu \circ (\id_A \widetilde{\otimes} \eta)
	\end{equation*}
	with $\widetilde{\otimes}$ denoting the completion of the projective topological tensor product. 
\end{dfn}
\begin{rmk}
	When the topology of the t.v.s. $A$ is generated by a family of seminorms $(p_{\alpha})_{\alpha \in \mathcal{F}}$, the topology on $A \widetilde{\otimes} A$ is generated by the family $(p_\alpha \otimes p_\beta)_{\alpha, \beta}$ of all possible tensor products of the $(p_{\alpha})_\alpha$ \cite{grothendieck_1955}. Recall each tensor product $p_\alpha \otimes p_\beta: A \widetilde{\otimes} A \to \mathbb{C}$ satisfies
\[
	(p_\alpha \otimes p_\beta)(z) = \inf_{(x_i), (y_i)} \sum_{1\leq i \leq m} p_\alpha(x_i) p_\beta(y_i).
\]
	for any $z = \sum_{1\leq i \leq m} x_i \otimes y_i$.
\end{rmk}

\begin{dfn}\cite[Definition~1]{Bonneau}
	A \textit{deformation} of a $\mathbb{C}$-algebra $A$ is a $\mathbb{C}[[h]]$-algebra $\tilde{A}$ such that $\tilde{A} / h A \cong A$. A deformation is \textit{trivial} if it is isomorphic to $A$ considered as a $\mathbb{C}[[h]]$-algebra by base field extension. Two deformations are \textit{equivalent} if they are isomorphic as $\mathbb{C}[[h]]$-algebras. 
\end{dfn}

Now set $A = \Ug[[h]]$. Then $A$ becomes a topological algebra when we equip it with the $h$-adic topology: since the topology is metrizable, the completion of the projective topological tensor product coincides with the completion of $A \otimes A$ in the $h$-adic topology; that is, 
\[
	A \widetilde{\otimes} A = (\Ug \otimes \Ug)[[h]].
\]

As topological \textit{algebras}, deformations of $\Ug$ are not very interesting: the next theorem shows that the multiplication on every deformation is essentially the same.
 	
\begin{thm}\label{rigidity of Lie alg}
 	Let $(A, \mu, \eta)$ be a deformation of $\Ug$ such that $A = \Ug[[h]]$ as a $\mathbb{C}[[h]]$-module and suppose the unit $\eta(1)$ is the constant formal series $1$ in $\Ug[[h]]$. There is a unique isomorphism $\alpha\colon A \to \Ug[[h]]$ of topological algebras inducing the identity $A / hA \to \Ug$ modulo $h$ and satisfying $\alpha = \id \mod h$.
	
	Moreover, if $\lieg'$ is any other semisimple Lie algebra and $\beta, \beta'$ are any two morphisms of topological algebras from $A$ to $U(\lieg')[[h]]$ such that $\beta \equiv \beta'$ modulo $h$, there exists an element $F$ in $U(\lieg')[[h]]$ such that $F \equiv 1$ modulo $h$ and $\beta'(a) = F \beta(a) \inv{F}$ for all elements $a$ in $A$. 
\end{thm}
\begin{proof}
	This theorem relies on vanishing results for certain cohomology groups associated to $\lieg$; see Sections~XVIII.2 and XVIII.3 in \cite{Kassel}. Drinfeld proves the uniqueness statement in \cite{Dri_QHA}.
\end{proof}

However, the category of Hopf algebras is less rigid; deformations of $\Ug$ are quite interesting as \textit{topological braided quasi-bialgebras}.

\begin{dfn}\cite[Section~XVI.4]{Kassel}\label[defn]{top braided bialgebra}
	A \textit{topological quasi-bialgebra} is a sextuple $(A, \mu, \eta, \Delta, \varepsilon, \Phi)$ such that the triple $(A, \mu, \eta)$ is a topological algebra over a ring $R$, $\Delta\colon A \to A \widetilde{\otimes} A$ and $\varepsilon\colon A \to R$ are continuous algebra maps, and $\Phi$ is an invertible element in $A \widetilde{\otimes} A \widetilde{\otimes} A$ satisfying 
	\begin{align*}
		(\varepsilon \widetilde{\otimes} \id_A) \Delta 
		= \id_A 
		= (\id_A \widetilde{\otimes} \varepsilon) \Delta,
		\quad
		(\id_A \widetilde{\otimes} \varepsilon \widetilde{\otimes} \id_A)(\Phi) 
		= 1 \widetilde{\otimes}  1,
		\quad \text{and}
	\end{align*}
	\begin{equation}\label[ax]{quasi pentagon}
		(\id_A \widetilde{\otimes}  \id_A \widetilde{\otimes} \Delta) (\Phi) (\Delta \widetilde{\otimes} \id_A \widetilde{\otimes} \id_A)(\Phi) 
		= \Phi_{234}(\id_A \widetilde{\otimes} \Delta \widetilde{\otimes}  \id_A)(\Phi) \Phi_{123}.
	\end{equation}
	We say $\Phi$ is an \textit{associator} and for all $a$ in $A$ we require
	\begin{equation}\label[ax]{axiom quasi coassoc}
		(\id_A \widetilde{\otimes} \Delta) (\Delta(a)) 
		= \Phi\big((\Delta \widetilde{\otimes}  \id_A) (\Delta(a)) \big) \inv{\Phi}.
	\end{equation}
	If $\Phi = 1 \widetilde{\otimes} 1 \widetilde{\otimes} 1$ is trivial, we say $A$ is a \textit{bialgebra}.
	
	We say the topological quasi-bialgebra $(A, \mu, \eta, \Delta, \varepsilon, \Phi)$ is \textit{braided} if there exists an invertible element $\Rmat$ in $A \widetilde{\otimes} A$ such that 
	\begin{align}
		(\id_A \widetilde{\otimes} \Delta) \Rmat 
			&= 
		\inv{(\Phi_{231})} \Rmat_{13} \Phi_{213} \Rmat_{12} \inv{(\Phi_{123})}, 
		\quad \text{and}
		\label[ax]{quasi qt 1} \\
		(\Delta \widetilde{\otimes} \id_A)(\Rmat) 
			&= 
		\Phi_{312} \Rmat_{13} \inv{(\Phi_{132})} \Rmat_{23} \Phi_{123},	
		\label[ax]{quasi qt 2}
	\end{align}
	We say $\Rmat$ is a \textit{universal R-matrix} and for all $a$ in $A$ we require
	\begin{align}\label[ax]{coprod quasitriangularity}
	\Delta^{\mathrm{op}}(a) = \Rmat \Delta(a) \inv{\Rmat}.
	\end{align}
\end{dfn}

Quasi-bialgebras were first introduced by Drinfeld in \cite{dri_89b}. They generalize the notion of a coassociative bialgebra by requiring coassociativity only up to conjugation by an associator. This is \Cref{axiom quasi coassoc}. Braided bialgebras generalize the notion of a cocommutative bialgebra in a similar way, by relaxing cocommutativity to \Cref{coprod quasitriangularity} instead.

At the level of categories, the representation theory of a braided quasi-bialgebra $A$ is captured by a braided category: \Cref{quasi pentagon} requires a consistency for the associator that is equivalent to the Pentagon \Cref{pentagon} and \Cref{quasi qt 1,quasi qt 2} enforce the Hexagon \Cref{hexagon1,hexagon2}. The category is strict only if $A$ is in fact a bialgebra.

We may now introduce \textit{quantized enveloping algebras}, as certain braided quasi-bialgebra deformations of $\Ug$.
\begin{dfn}\cite[Definition~XVI.5.1]{Kassel}\label{que def}
	A \textit{quantized enveloping algebra (QUE)} for the Lie algebra $\lieg$ is a topological braided quasi-bialgebra $A = (A, \mu, \eta, \Delta, \varepsilon, \Phi, \Rmat)$ such that 
	\begin{enumerate}
		\item $A$ is a topologically free module, which means $A / h A = \Ug$ as a left $\mathbb{C}[[h]]$-module.
		\item For $a, a' \in \Ug$, we have 
		\begin{align*}
		\mu(a \otimes a') 
			&= \mu_0(a \otimes a') + \sum_{n \geq 1}\mu_n(a \otimes a') h^n, \\
		\Delta(a) 
			&= \Delta_0(a \otimes a') + \sum_{n \geq 1}\Delta_n(a) h^n, 
			\quad \text{and} \\
		\varepsilon(a) 
			&= \varepsilon_0(a) + \sum_{n \geq 1}\varepsilon_n(a) h^n, 
		\end{align*}
		where $\mu_0$, $\Delta_0$, and $\varepsilon_0$ denote the multiplication, comultiplication, and counit maps of the enveloping algebra $\Ug$, and $(\mu_n)_{n\geq 0}$, $(\Delta_n)_{n\geq 0}$, $(\varepsilon_n)_{n\geq 0}$ are families of linear maps from $\Ug \otimes \Ug$ to $\Ug$, $\Ug$ to $\Ug \otimes \Ug$, and $\Ug$ to $\mathbb{C}$, respectively. 
		\item The unit $\eta$ extends the unit of $\Ug$ trivially, so $\eta(f) = f 1$ for all $f \in \mathbb{C}[[h]]$.
		\item The elements $\Phi$ and $\Rmat$ can be written as 
		\begin{equation*}
		\Phi = 1 \otimes 1 \otimes 1 + \sum_{n \geq 1} \Phi_n h^n, 
		\quad \text{and} \quad
		\Rmat = 1 \otimes 1 + \sum_{n \geq 1} R_n h^n,
		\end{equation*}
		with $(\Phi_n)_{n\geq 0 }$ and $(R_n)_{n\geq 0 }$ denoting elements of $\Ug^{\otimes 3}$ and $\Ug^{\otimes 2}$.
	\end{enumerate}
\end{dfn}

There is an important invariant associated to a QUE. If $A$ is a QUE with $R$-matrix $\Rmat$, the formula
\begin{equation}\label[rel]{defn canonical 2 tensor}
\Rmat_{21}\Rmat \equiv 1 \otimes 1 + 2h t \mod h^2
\end{equation}
defines a unique element $t \in \Ug \otimes \Ug$. It turns out that $t$ is a symmetric in variant $2$-tensor on $\Ug$ in the sense of \Cref{symm inv 2-tensor}.

\begin{dfn}\label[defn]{canonical 2 tensor and quantization}
	The \textit{canonical $2$-tensor} of a QUE $A$ is the element $t \in \lieg \otimes \lieg$ defined by \Cref{defn canonical 2 tensor}. The pair $(\lieg, t)$ is called the \textit{classical limit} of $A$. Dually, $A$ is a \textit{quantization} of $(\lieg, t)$.
\end{dfn}

In our context, the canonical $2$-tensor is crucial because it characterizes the (action of the) braiding on a QUE. \Cref{every que is Agtheta} makes this precise, but it requires the notion of \textit{gauge equivalence}.

\begin{dfn}\cite[Definition~XV.3.1]{Kassel}\label[defn]{gauge transform}
	A \textit{gauge transformation} on the topological quasi-bialgebra $A$ is an invertible element $F$ of $A \widetilde{\otimes} A$ such that
	\begin{equation*}
	(\varepsilon \widetilde{\otimes} \id)(F) = (\id \widetilde{\otimes} \varepsilon)(F) = 1.
	\end{equation*}
\end{dfn}
\begin{rmk}
	Gauge transforms are sometimes referred to as \textit{Drinfeld twists} and less frequently as \textit{skrooching} in the literature, e.g. in \cite{Dri_QHA_KZ_eqns} and \cite{stasheff}. The term ``skrooching'' results from the transliteration of a Russian term for twisting.
\end{rmk}
Notice that twisting defines an equivalence relation: $(A_F)_{\inv{F}} = A = (A_{\inv{F}})_F$ and if $F$ and $F'$ are gauge transformations, then $FF'$ defines a gauge transformation such that $(A_F)_{F'} = A_{FF'}$. Therefore, we say that two (braided) quasi-bialgebras $(A, \Delta, \varepsilon, \Phi)$ and $(A', \Delta', \varepsilon', \Phi')$ are \textit{equivalent} if there exists a gauge transform $F$ on $A'$ and an isomorphism $\alpha\colon A \to A_F'$ of (braided) quasi-bialgebras.

We use gauge transformations to obtain new QUEs from old ones. In particular, if $A = (A, \mu, \eta, \Delta, \varepsilon, \Phi)$ is a topological quasi-bialgebra and $F$ is a gauge transformation on $A$, we obtain another topological quasi-bialgebra $A_F = (A, \mu, \eta, \Delta_F, \varepsilon, \Phi_F)$, with 
\begin{equation*}
\Delta_F(a) = F \Delta(a) \inv{F}, \,\, a \in A,
\quad\text{and}\quad
\Phi_F = F_{23}(\id \widetilde{\otimes} \Delta)(F) \Phi (\Delta \widetilde{\otimes} \id)(\inv{F}) \inv{F}_{12}
\end{equation*}
by ``twisting'' $A$ \cite[Proposition~XV.3.2]{Kassel}.

Moreover, if $A$ is braided with universal $R$-matrix $\Rmat$, then so is $A_F$, with $R$-matrix $\Rmat_F = F_{21} \Rmat \inv{F}$ \cite[Proposition~XV.3.6]{Kassel}. This means canonical $2$-tensors are invariant under gauge transformation:
\[
	(\Rmat_F)_{21} \Rmat_F 
	= F \Rmat_{21} \Rmat \inv{F} 
	= 1 \otimes 1 + 2h t  \mod h^2.
\]
It also means the braiding operators $\check{R}$ and $\check{R}_F$ induced by $A$ and $A_F$ are similar:
\begin{equation}\label{braiding similarity}
	\check{R}_F = \tau \circ \Rmat_F = F \check{R} \inv{F}.
\end{equation}

\Cref{braiding similarity} implies that if $A$ and $A'$ are equivalent topological braided quasi-bialgebras and $M$ is any \textit{finite-rank topologically free} module over $A$, that is, a module of the form $M = V[[h]]$ with $V$ finite-dimensional, then $A$ and $A'$ induce equivalent braid group representations on $M^{\otimes n}$. 

It turns out the braid representations induced by any QUE are equivalent to those induced by a QUE of the form
\begin{equation*}
	A_{\lieg, \theta} = (\Ug[[h]], \Delta_0, \varepsilon_0, \Phi, \Rmat_\theta = e^{h\theta}),
\end{equation*}
where $\theta$ is a symmetric invariant $2$-tensor on $\Ug$, in the sense of \Cref{symm inv 2-tensor}. Notice $A_{\lieg, \theta}$ is a quantization of the pair $(\lieg, \theta)$, in the sense of \Cref{canonical 2 tensor and quantization}, and it is equipped with the undeformed multiplication and comultiplication of $\Ug$. While $A_{\lieg, \theta}$ is cocommutative because the invariance of $\theta$ guarantees $R_\theta$ commutes with $\Delta_0(\Ug)$, it is not coassociative in general.

\begin{thm}\cite[Theorem~1]{Dri_QHA_KZ_eqns}\label{every que is Agtheta}
	If $(A, \Delta, \Phi, \Rmat)$ is a QUE, there exists an associator $\Phi \in \Ug[[h]]^{\otimes 3}$ and a symmetric invariant $2$-tensor $\theta$ on $\Ug$ such that $A$ is equivalent to $A_{\lieg, \theta}$ as a topological braided quasi-bialgebra. If $\Rmat$ is known explicitly, then $\theta$ is the limit of $\inv{h}(\Rmat_{21}\Rmat - 1)$ as $h \to 0$.
\end{thm}

At the categorical level, this result implies that for any QUE $A$ there is a symmetric invariant $2$-tensor $\theta$ on $\Ug$ such that $\Amodcat$ and $A_{\lieg, \theta}\operatorname{-Mod_{fr}}$ are tensor equivalent, since the categories of modules over equivalent QUEs are tensor equivalent \cite[Theorems~XV.3.5,XV.3.9]{Kassel}.

We obtain a geometric interpretation of the Drinfeld-Kohno \Cref{every que is Agtheta} via the \textit{KZ equations}; see e.g. \cite{kz,kazhdan_lusztig_1993,finkelberg_1996}.

\begin{dfn}
	Given a symmetric invariant $2$-tensor $\theta$ on $\Ug$ as in \Cref{symm inv 2-tensor}, a complex parameter $h$, a finite-dimensional $\lieg$-module $V$, and $n > 1$, the \textit{Knizhnik-Zamolodchikov (KZ) differential system} is given by
	\begin{equation}\label{KZ eqns}
	dw = \frac{h}{2\pi \sqrt{-1}} \sum_{1 \leq i < j \leq n} \frac{\theta_{ij}}{z_i - z_j}(dz_i - dz_j)w, 
	\end{equation}
	where $w = w(z_1, \ldots, z_n)$ is a function on (ordered) configuration space with values in $V^{\otimes n}$.
\end{dfn}

In particular, consider a QUE $A$, let $\theta$ denote its canonical $2$-tensor, and fix any topologically free finite-rank $A$-module $V[[h]]$. Recall \Cref{symm inv tensors satisfy inf braid rel}, which shows that the map $t_{ij} \to \theta_{ij}$ defines a representation $\mathfrak{p}_n \to \End(\Vtn)$ of the infinitesimal braid algebra, and note that the fundamental group of (ordered) configuration space is the \textit{pure braid group} $P_n$, defined as the kernel of the projection $B_n \to S_n$ of the braid group onto the symmetric group \cite[Remark~XIX.2.3]{Kassel}. The KZ equations define a flat connection on $\Vtn$, so they induce a monodromy representation of $P_n$ on $V^{\otimes n}$. We may therefore view the inifinitesimal braid algebra, with Lie bracket given by the commutator, as the Lie algebra analogue of a Lie group for the pure braid group $P_n$; monodromy is the analogue of integrating a representation of the Lie algebra to a representation of the Lie group. In fact, infinitesimal braids were first introduced by Kohno \cite{Kohno} to explain the monodromy of the KZ equations as they arise in Wess-Zumino-Witten conformal field theories (CFTs).

The monodromy representation of $P_n$ may be extended to a representation of the full braid group and Drinfeld shows this (analytic) braid representation is equivalent to that induced by the $R$-matrix $\Rmat_\theta = e^{h \theta}$, once we identify analytic functions with formal series~\cite{Dri_QHA_KZ_eqns}. In view of \Cref{every que is Agtheta}, this representation is in turn equivalent to the representation induced by $A$ on $V[[h]]^{\otimes n}$. 

Moreover, Drinfeld constructs the associator $\Phi$ mentioned in \Cref{every que is Agtheta} explicitly using solutions to KZ system~\cite{Dri_QHA_KZ_eqns}. Drinfeld's construction shows that, conversely, solutions to the KZ equations give rise to braided quasi-bialgebras.

Now we turn our attention to the QUE $\Uh$ defined below. 
\begin{dfn}\cite[Definition~XVII.2.3]{Kassel}\label[defn]{defn:Uh}
Consider the enveloping algebra $\Ug$ as described by \Cref{chevalley presentation of g,Serre rels of g}. Set $q_i = e^{hd_i}$ and fix $q = e^h$. Let $\Uh$ denote the unital associative $\mathbb{C}[[h]]$-algebra generated by the $\Ug$ generators $E_1, \ldots, E_r$, $F_1, \ldots, F_r$, and $H_i, \ldots, H_r$ subject to the relations
	\begin{equation}\label[rel]{Uh comm rels}
		\begin{gathered}
		[H_i, H_j]  = 0, 
		\quad 
		[H_i, E_j] = a_{ij} E_j, 
		\quad 
		[H_i, F_j] = -a_{ij} F_j,
		\quad \text{and} \\
		[E_i, F_j] = \delta_{ij} \frac{e^{hd_iH_i} - e^{-hd_iH_i}}{q_i - \inv{q_i}},
		\end{gathered}
	\end{equation}
	along with the $q$-Serre relations for $i\neq j$:
	\begin{equation}\label[rel]{q serre for Uh}
		\sum_{k = 0}^{1-a_{ij}} (-1)^k 
		\begin{bmatrix}
		1 - a_{ij} \\ k
		\end{bmatrix}_{q_i} X_i^k X_j X_i^{1-a_{ij}-k} = 0.
	\end{equation}
	In the last equality every $X$ is either an $E$ or an $F$. The $q$-integer $[n]_{q_i} = \frac{q_i^n - q_i^{-n}}{q_i - \inv{q_i}}$, the $q$-factorial $[n]_{q_i}! = [n]_{q_i}\cdots [1]_{q_i}$, and the $q$-binomial coefficient is defined similarly.
\end{dfn}

Observe there is an algebra isomorphism $\Uh / h \Uh \cong \Ug$ arising as follows. First note that while $q_i - \inv{q_i} = 2\sinh(hd_i) = 2\sum_{n \geq 0} \frac{(hd_i)^{2n+1}}{(2n+1)!}$ is not invertible $\mathbb{C}[[h]]$, it is the product of $h$ and a unit. This means $[E_i, F_j]$ is indeed well-defined. In particular, 
\begin{equation*}
	[E_i, F_j] = \delta_{ij} \frac{\sinh(hd_iH_i)}{\sinh(hd_i)} = \delta_{ij} H_i \mod h,
\end{equation*}
so we recover \Cref{chevalley presentation of g,Serre rels of g}, characterizing $\Ug$, by setting $h = 0$ in \Cref{Uh comm rels,q serre for Uh}. 

In \cite{Dri_qgps}, Drinfeld proved that $\Uh$ is in fact a braided bialgebra with universal $R$-matrix of the form 
\begin{equation}\label{R matrix of Uh}
	\Rmat_h = \exp\bigg(h\sum_{1 \leq i, j \leq r} (D\inv{A})_{ij} H_i \otimes H_j\bigg) \sum_{\ell \in \mathbb{N}^r} P_\ell,
\end{equation}
where $P_\ell$ is a homogeneous polynomial of degree $\ell_i$ in the variables $E_i \otimes 1$ and in $1 \otimes F_i$. In particular, $P_0 = 1 \otimes 1$ so $\Rmat_h \equiv 1 \otimes 1 \mod h$. Explicit formulas for $\Rmat_h$ may be found in, e.g., \cite{kr90,rosso_qgps}. The QUE $\Uh$ is coassociative but not cocommutative, because the associator is trivial but the $R$-matrix is not. The comultiplication $\Delta_h$ and the counit $\varepsilon_h$ satisfy
\[
	\Delta_h = \Delta_0 \mod h 
	\quad \text{and} \quad 
	\varepsilon_h = \varepsilon_0 \mod h.
\]
For details, see \cite[Theorem~8.3.9]{chari_pressley_1994}.

The QUE $\Uh$ stands out because its canonical $2$-tensor $t^*$ is induced by the Casimir element $C$ in $\Ug$ through the formula
\begin{equation}\label{uh canonical tensor}
	t^* = {1 \over 2}(\Delta_0(C) - C \otimes 1 - 1 \otimes C).
\end{equation}
This can be seen by considering $\Uh$ as a \textit{ribbon algebra}.

\begin{dfn}\label[defn]{ribbon algebra}
	A (topological) braided Hopf algebra $(H, \Delta, \varepsilon, S, \Rmat)$ is a \textit{(topological) ribbon algebra} if there exists a central element $\theta$ in $H$ satisfying
	\begin{equation*}
	\Delta(\theta) = \inv{(\Rmat_{21} \Rmat)}(\theta \otimes \theta), \quad \varepsilon(\theta) = 1, \quad S(\theta) = \theta.
	\end{equation*}	 	
	We say $\theta$ is a \textit{ribbon element}.
\end{dfn}
\begin{rmk}\label[rmk]{ribbon cat if ribbon alg}
	If $H$ is a ribbon algebra with ribbon element $\theta$, then $\Hmodcat$ is a ribbon category with twist $\theta_V$ satisfying $\theta_V(v) = \inv{\theta} v$ for every $V \in \Hmodcat$ \cite[Proposition~XIV.6.2]{Kassel}. Moreover, for any $V, W \in \Hmodcat$, the induced braiding operator $c_{V, W} = \tau \circ \Rmat$ satisfies
	\begin{equation}\label{braiding sq in terms of twist}
		c_{V, V}^2 = (\Rmat_{21} \Rmat) = \Delta(\theta) (\theta \otimes \theta).
	\end{equation}
\end{rmk}

\begin{thm}\label{R matrix of Uh special relation}
	Let $C$ denote the Casimir element of $\Ug$, as in \Cref{casimir elt}, let $\alpha\colon A \to \Ug[[h]]$ denote the unique isomorphism satisfying $\alpha \equiv \id \mod h$ guaranteed by \Cref{rigidity of Lie alg}, and set $C_h = \inv{\alpha}(C)$. Then the QUE $\Uh$ is a ribbon algebra with ribbon element $\theta_h = e^{-h C_h}$.
\end{thm}

A key step in the proof of \Cref{R matrix of Uh special relation} involves identifying central elements in $\Uh$ by comparing their actions on certain modules. This is possible because: $(1)$ $\alpha$ maps the center of $\Uh$ to $Z(\lieg)[[h]]$ isomorphically, with $Z(\lieg)$ denoting the center of $\Ug$, thereby identifying each central element in $\Uh$ with a polynomial function on the weight lattice of $\lieg$ via the classical Harish-Chandra homomorphism; and $(2)$ the representation theory of $\Uh$ is parallel to that of $\Ug$.

Indeed, there is a one-to-one correspondence between irreducible $\Ug$-modules and indecomposable finite-rank topologically free modules over $\Uh$. Over $\Uh$, we consider indecomposable, rather than irreducible, modules because the $\Uh$-action preserves each subspace $h^n V[[h]]$ in any topologically free $V[[h]]$. In particular, for every dominant weight $\mu$ of $\lieg$ there is a unique indecomposable topologically free $\Uh$-module $\widetilde{V}(\mu)$ satisfying
\begin{equation*}
\widetilde{V}(\mu) / h \widetilde{V}(\mu) \cong L(\mu),
\end{equation*}
with $L(\mu)$ denoting an irreducible $\Ug$-module with highest weight $\mu$ \cite{Dri_QHA}. As in the classical case, there is a \textit{highest weight vector} $v_{\mu}$ in $\widetilde{V}(\mu)$ satisfying
\begin{equation*}
H_i v_{\mu} = \mu(H_i) v_{\mu} \quad \text{and} \quad E_i v_{\mu} = 0,
\end{equation*}
for $i = 1, \ldots, r$, that both generates and characterizes $\widetilde{V}(\mu)$ \cite[Proposition~4.3]{Dri_QHA}.

By uniqueness, the $\Uh$-action on $\widetilde{V}(\mu)$ must coincide with the composition 
\[
	\Uh \xrightarrow{\,\smash{{\ensuremath{\scriptstyle\alpha}}}\,} \Ug[[h]] \to \End(V[[h]]).
\] 
Letting $\chi_{\mu}(C) = \langle \mu, \mu + 2\rho \rangle$ denote the eigenvalue of $C$ on $L(\mu)$ as in \Cref{casimir eigenvalues}, it follows that the element $\theta_h$ acts on $\widetilde{V}(\mu)$ as
\begin{equation}\label{ribbon elt action}
	\theta_h = e^{-h\alpha(C_h)} = e^{-hC} = q^{-\langle \mu, \mu + 2\rho \rangle}.
\end{equation} 

\begin{proof}[Proof of \Cref{R matrix of Uh special relation}]
	Let $S_h$ denote the antipode of $\Uh$. Set $u = \mu((S_h\otimes \id)\Rmat_{21})$ and take $\rho^* = \sum_{i} b_i H_i$ as in \Cref{weyl vector}. We show $e^{-2h \rho^*} u$ is a ribbon element and conclude by showing it acts on $\widetilde{V}(\mu)$ as in \Cref{ribbon elt action}. 
	
	First we show $\theta$ is central by comparing two expressions for $S_h^2$. On one hand, Proposition~2.1 in \cite{Dri_QHA} shows $S_h^2(x) = u x \inv{u}$. On the other, we use Lemma~6.4.2 in \cite{chari_pressley_1994} to prove $e^{h z H_j} X_i^\pm e^{-h z H_j} = e^{\pm h z \langle \alpha_j^\vee, \alpha_i \rangle} X_i^\pm$ with $X_i^+ = E_i$ and $X_i^- = F_i$, which implies $S_h^2(x) = e^{2h\rho^*} x e^{-2h\rho^*}$. Rearranging $u x \inv{u} = S_h^2(x) = e^{2h\rho^*} x e^{-2h\rho^*}$ implies 
	$$\theta x = (e^{-2h\rho^*} u) x = x (e^{-2h\rho^*} u) = x\theta.$$
	
	Next, applying the two expressions for $S_h^2$ with $x = u$ yields $e^{-2h \rho^*} u = u e^{-2h \rho^*}$. Propositions~3.1 and $3.2$ in \cite{Dri_QHA} prove $\varepsilon(u) = 1$ and $\Delta(u) = \inv{(\Rmat_{21}\Rmat)} (u \otimes u)$. Therefore $\varepsilon(\theta_h) = \varepsilon(u) = e^{-2h \varepsilon(\rho^*)} = 1$, and
	$$
		\Delta_h(\theta) 
		= \Delta_h(u) \Delta_h(e^{-2h\rho^*}) 
		= \inv{((\Rmat_h)_{21}\Rmat_h)} (u e^{-2h\rho^*} \otimes u e^{-2h\rho^*}) 
		= \inv{((\Rmat_h)_{21}\Rmat_h)}(\theta\otimes \theta).
	$$
	
	To conclude, we compute the action of $e^{-2h \rho^*} u$ on $\widetilde{V}(\mu)$. Since $e^{-2h \rho^*} u$ is central, it suffices to consider a highest weight vector $v_\mu$. For starters, we compute that
	\begin{align*}
		e^{-2 h \rho^*} v_\mu
		= \prod_i \left(\sum_{k \geq 0} (-2h)^k b_i^k \frac{\mu(H_i)^k}{k!}\right) v_{\mu} 
		= \exp\left(-2h\sum_{i} b_i \mu(H_i)\right) v_\mu
		= e^{-h \langle \mu, 2\rho \rangle} v_\mu.
	\end{align*}
	Now notice the formula defining $u$ is such that every $E_i$ appears on the right. Each $P_\ell$ in \eqref{R matrix of Uh} is homogeneous of degree $\ell_i$ in $E_i \otimes 1$ and each $E_i$ annihilates $v_\mu$, so only the $\ell = 0$ term contributes to the sum. Therefore,
	\begin{align}\label{action of u}
	\begin{split}
		u \rhd v_\mu &= \exp\big(h \sum_{i,j} (D\inv{A})_{ij} S(H_j) H_i\big) v_\mu\\
		&= \prod_{i, j} \bigg(\sum_{k \geq 0} (-h)^k {(D \inv{A})}_{ij}^{k} \frac{\mu(H_j)^k \mu(H_i)^k}{k!}\bigg) v_\mu \\
		&= \exp\big(-h \sum_{i,j} {(D\inv{A})}_{ij} \mu(H_j) \mu(H_i)\big) v_\mu \\
		&= e^{-h \langle \mu, \mu \rangle} v_\mu.		
	\end{split}
	\end{align}
	The last equality follows immediately from \Cref{scalar product of weights}. Combining results implies $e^{-2 h \rho^*} u$ indeed acts on $\widetilde{V}(\mu)$ as in \Cref{ribbon elt action}, so $\theta_h = e^{-2 h \rho^*} u$. 
	
	Notice $\theta_h$ is indeed fixed by the antipode: $S_h(\theta_h)$ acts on $\widetilde{V}(\mu)$ as $\theta_h$ acts on the contragradient $\widetilde{V}(\mu)^*$, whose highest weight is the opposite of the reflection of $\mu$ by the long element $w_0$ in the Weyl group $W$. Thus $S_h(\theta_h)$ acts on $\widetilde{V}(\mu)$ as the scalar
\[
	e^{-h \langle -w_0 \lambda, -w_0 \lambda \rangle} e^{-h \langle -w_0 \lambda, 2\rho \rangle} = e^{-h \langle \lambda, \lambda + 2\rho\rangle},
\]
		because $-w_0 \rho = \rho$ and the form is $W$-invariant \cite{leduc_ram_1997}.
\end{proof}

Notice that \Cref{R matrix of Uh special relation} implies 
\begin{align*}
\begin{split}
	(\Rmat_h)_{21} \Rmat_h 
	&= \Delta_h(\theta_h) (\theta_h \otimes \theta_h) 
	= e^{h \Delta_h(C_h)} (e^{-hC_h} \otimes e^{-hC_h}) \\
	&\equiv 1 \otimes 1 + \big(\Delta_0(C) - C \otimes 1 - 1 \otimes C\big)h \mod h^2.
\end{split}
\end{align*}
The last identity uses the congruence $C_h \equiv C \mod h$.

Specializing \Cref{every que is Agtheta} to the case $A = \Uh$ then shows $\Uh$ is equivalent to a QUE of the form
\begin{equation}\label{que for uh}
	(\Ug[[h]], \Delta, \varepsilon, \Phi, \Rmat = e^{ht^*}).
\end{equation}
This means the braiding on $\Uh$ is generated infinitesimally by the symmetric invariant $2$-tensor $t^*$ induced by the Casimir element on $\Ug$ via \Cref{uh canonical tensor}. 

This result is key. The following theorem shows $\Uh\operatorname{-Mod_{fr}}$ essentially characterizes all possible \textit{$h$-adic deformations} of $\Ugmodcat$. 

\begin{dfn}\cite[Definition~3.5]{yetter_1992}\label[defn]{defn:deformation cat}
	An $h$-adic \textit{deformation} of a braided $\mathbb{C}$-linear category $\ccat$ is a $\mathbb{C}[[h]]$-linear category with the same objects as $\ccat$ but with the $\mathbb{C}[[h]]$-module hom-sets from $V$ to $W$ given by $\mathbb{C}[[h]] \widetilde{\otimes} \Hom(V, W)$, with composition extended $\mathbb{C}[[h]]$-bilinearly and continuously. The deformation is equipped with a monoidal structure that matches the structure on $\ccat$ on objects, and extends it $\mathbb{C}[[h]]$-bilinearly and continuously on maps. The associativity, left and right unit constraints, and the braiding on the deformation all reduce modulo $h$ to the corresponding maps on $\ccat$.
\end{dfn}

\begin{thm}\cite[Example~7.22.6]{etingof_gelaki_nikshych_ostrik_2017}\cite[Example~2.24]{calaque_etingof_2004}\label{deformation cat}
	Let $G$ be a complex reductive Lie group with Lie algebra $\lieg$, and let $\mathrm{Rep}(G)_f$ denote the category of finite-dimensional $G$-modules. There exists a unique one-parameter deformation of $\mathrm{Rep}(G)_f$, and it is realized by $\Uh\operatorname{-Mod_{fr}}$.
\end{thm}
	
	\subsection{$\Uqgmodcat$ is a ribbon category}
	\label{uqg_modcat}
	\longer{Recall \ref{not and conv}.}

We show the category $\Uqgmodcat$ of finite-dimensional modules over the quantum group $\Uqg$ is a ribbon category; by \Cref{deformation cat}, this category is tensor equivalent to $\Uh\operatorname{-Mod_{fr}}$. This implies the braid group representation induced by the braiding on $\Uqgmodcat$ on any tensor product $V^{\otimes m}$ is equivalent to the representation induced by the $R$-matrix of $\Uh$, once we identify $q = e^h$. In view of \Cref{que for uh}, this means the representations induced by $\Uqgmodcat$ are generated infinitesimally by the symmetric invariant $2$-tensor $t$ in \Cref{uh canonical tensor}, which is induced by the Casimir element in $\Ug$. Concretely, if $V$ is any simple $\Uqg$-module, the braid representation $\sigma_i \to \check{R}_i$ on $V^{\otimes m}$ induced by the braiding on $\Uqgmodcat$ is equivalent to the representation induced by $q^t$.

For starters we describe the objects in $\Uqgmodcat$. When $q$ is transcendental over $\mathbb{Q}$, the finite-dimensional representation theory of $\Uqg$ is parallel to that of $\Ug$: there is a one-to-one correspondence between simple modules, as follows. Every finite-dimensional $\Uqg$-module is a direct sum of \textit{weight spaces} $V_\mu$ defined by 
\[
	V_\mu 
	= \{ 
		v \in V 
		\mid 
		K_i v = q^{\langle \mu, \alpha_i \rangle} v, \,
		i = 1, \ldots, r
	  \}
\]
for any weight $\mu$ in $\mathcal{P}$ \cite[Proposition~5.1]{Jantzen96}. In addition, every finite-dimensional $\Uqg$-module is completely reducible. Every finite-dimensional irreducible module $V$ is a \textit{highest weight module}, which means there is a dominant weight $\mu \in \mathcal{P}^+$ and $v \in V_\mu$ such that $E_i v = 0$ for $i = 1, \ldots, r$ and $V = \Uqg \cdot v$ \cite[Theorem~5.10]{Jantzen96}. Conversely, every highest weight module is in fact irreducible \cite[Theorem~10.1.6]{chari_pressley_1994}. Thus for each $\mu \in \mathcal{P}^+$, there is exactly one irreducible $\Uqg$-module $V(\mu)$, which decomposes as $V(\mu) = \bigoplus_\nu V_\nu$, with the same weight multiplicities as in an irreducible $\Ug$-module with highest weight $\mu$ \cite[Theorem~4.12]{lusztig_1988}.

\begin{rmk}
	Our discussion characterizes $\Uqg$-modules \textit{of type} $\mathbf{1} = (1, \ldots, 1)$. Weight spaces can be defined more generally, by taking any group homomorphism $\sigma\colon \mathcal{Q} \to \{\pm 1\}$ and setting $V_{\mu, \sigma} 
	= \{ 
		v \in V 
		\mid 
		K_i v = \sigma(\alpha_i) q^{\langle \mu, \alpha_i \rangle} v, \,
		i = 1, \ldots, r
	  \}$. This more general definition results in modules \textit{of type} $\sigma$; see Chapter~5 in \cite{Jantzen96} or Chapter~9 in \cite{chari_pressley_1994} for details.
\end{rmk}

Next we show $\Uqgmodcat$ is a braided category. We construct a braiding on $\Uqgmodcat$ in two steps. First recall \Cref{R matrix of Uh}, which defines the $R$-matrix $\Rmat_h$ of $\Uh$, and write $\Rmat_h = \mathcal{E}_h \mathcal{S}_h$ with 
$$ 	
	\mathcal{E}_h = \exp\bigg(h\sum_{1 \leq i, j \leq r} (D\inv{A})_{ij} H_i \otimes H_j\bigg)
	\qquad \text{and} \qquad
	\mathcal{S}_h = \sum_\ell P_\ell.
$$
Each $E_i$, $F_i$ acts nilpotently on every finite-dimensional $\Uqg$-module, so all but finitely many summands of $\mathcal{S}_h$ act trivially on any tensor product. This means $\mathcal{S}_h$ induces a well-defined endomorphism on every tensor product $V(\mu) \otimes V(\nu)$ of simple objects in $\Uqgmodcat$. Call this map $R_{\mu, \nu}$. 

The exponential term is more nuanced. For any pair $V(\mu)$, $V(\nu)$ of irreducible modules, $V(\mu) \otimes V(\nu)$ is a direct sum of subspaces of the form $V_{\mu'} \otimes V_{\nu'}$. A calculation like the one in \Cref{action of u} shows each $V_{\mu'} \otimes V_{\nu'} \subseteq V(\mu) \otimes V(\nu)$ is an eigenspace for $\mathcal{E}_h$ with eigenvalue $e^{h \langle \mu', \nu' \rangle}$. For this reason we define $E_{\mu, \nu}$ as the operator on $V(\mu) \otimes V(\nu)$ acting as the scalar $q^{\langle \mu', \nu' \rangle}$ on the subspace $V_{\mu'} \otimes V_{\nu'}$.

We compose these two endomorphisms with the flip map $\tau$ taking $v \otimes w \to w \otimes v$ as usual to obtain a morphism $c_{\mu, \nu}\colon V(\mu) \otimes V(\nu) \to V(\nu) \otimes V(\mu)$ given by
$$c_{\mu, \nu} = \tau \circ E_{\mu, \nu} R_{\mu, \nu}.$$
Proposition~10.1.19 in \cite{chari_pressley_1994} proves the family $(c_{\mu, \nu})_{\mu, \nu}$ defines a braiding on $\Uqgmodcat$.

If we identify $q = e^h$, we see that the actions of $E_{\mu, \nu} R_{\mu, \nu}$ and $\Rmat_h$ on $V(\mu) \otimes V(\nu)$ are identical: the map $E_{\mu, \nu} R_{\mu, \nu}$ is defined to \textit{implement} $\Rmat_h$. Technically, this achieves this section's goal: we have just shown that the action of the braiding on $\Uqgmodcat$ on any tensor product $V(\mu) \otimes V(\nu)$ coincides with the action of the braiding operator induced by the $R$-matrix of $\Uh$, which is generated infinitesimally by the canonical tensor defined by \Cref{uh canonical tensor}. Nice!

However, this method does not tell the whole story. The category $\Uqgmodcat$ is in fact a ribbon category that may be realized as a category of modules over a certain ribbon algebra $\widetilde{\Uqg}$. This characterization of $\Uqgmodcat$ is important because it yields a convenient expression for the (square of) the braiding operators in terms of the algebra's ribbon element via \Cref{braiding sq in terms of twist}.

The algebra $\widetilde{\Uqg}$ is a completion of $\Uqg$ whose elements may be characterized by their action on every finite-dimensional $\Uqg$-module. We dedicate the rest of this section to constructing $\widetilde{\Uqg}$ and exploring some of its basic properties. Our discussion culminates in \Cref{Uq is a ribbon alg}. 

To begin let $U_q^-, U_q^0$, and $U_q^+$ denote the subalgebras of $\Uqg$ generated by the $E_i$, the $K_i^\pm$, and the $F_i$. Let $(U_q^+)_\beta = \{ u \in U_q^+ \, | \, K_i u \inv{K_i} = q^{ \langle \beta, \alpha_i \rangle} u, \, i = 1, \ldots, n\}$ and define the \textit{height} of a root $\beta = \sum_i \beta_i \alpha_i$ as $\mathrm{height}(\beta) = \sum_i \beta_i$. For each $m \geq 1$, set
\begin{equation*}
	U_q^{+, m} = \bigoplus_{\beta \in \mathcal{Q}: \,\mathrm{height}(\beta) \geq m} (U_q^+)_\beta.
\end{equation*}	 
The completion $\widetilde{\Uqg}$ is defined as the inverse limit
\begin{equation}\label[rel]{completion of Uq}
	\widetilde{\Uqg} =  \varprojlim (\Uqg) / (\Uqg) U_q^{+, m}.
\end{equation} 

This algebra can be understood as a direct product because $\Uqg$ has a triangular decomposition and a PBW-type basis, much like its classical counterpart $\Ug$ \cite[Theorem~6.14]{klimyk_schmudgen_1997}. Since 
\begin{equation*}
	\Uqg \cong \bigoplus_{\beta \in \mathcal{Q}^+} U_q^{\leq 0} (U_q^+)_\beta,
\end{equation*}
with $U_q^{\leq 0}$ denoting $U_q^- U_q^0$, it follows that $\widetilde{\Uqg}$ is the direct product
\begin{equation*}
	\widetilde{\Uqg} \cong \prod_{\beta \in \mathcal{Q}^+} U_q^{\leq 0} (U_q^+)_\beta
\end{equation*}
whose elements are infinite formal sums $\sum_\beta x_\beta$ with $x_\beta \in U_q^{\leq 0} (U_q^+)_\beta$ \cite[6.3.3]{klimyk_schmudgen_1997}. 

The direct product suggests an alternative characterization of $\widetilde{\Uqg}$. 
 
\begin{dfn}\cite[Definition~3.1]{kamnitzer_tingley_2008}\label[defn]{kt completion of Uq}
	Let $S$ denote the ring of formal series $\sum_{k=1}^\infty x_k$ such that each $x_k \in \Uqg$ and for any fixed $\mu$, we have $x_k \rhd V(\mu) = 0$ for all but finitely many $k$. Let $I$ denote the two-sided ideal of $S$ consisting of elements which act as zero on all $V(\mu)$. Then $\widetilde{\Uqg}$ is defined as $S / I$. 
\end{dfn}

Indeed, the two completions just defined are equivalent. In fact each element in either completion is specified uniquely by its action on every $V(\mu)$ in $\Uqgmodcat$.

\begin{thm}\label{equivalence of completions}
	Let $\widetilde{\Uqg}$ denote the completion of $\Uqg$ defined by \labelcref{kt completion of Uq} or by \Cref{completion of Uq}. There is an isomorphism of $\mathbb{C}(q)$-algebras
	$$\widetilde{\Uqg} \cong \prod_{\mu \in \mathcal{P}^+} \End_{\mathbb{C}(q)}(V(\mu)).$$
\end{thm}	
This result combines Theorem $3.3$ and Corollary $3.6$ in \cite{kamnitzer_tingley_2008}. Its proof relies on the following technical lemma.

\begin{lem}\cite[Lemma~3.4]{kamnitzer_tingley_2008}\label[lem]{defn p lambda}
	For each $\mu \in \mathcal{P}^+$, fix a highest weight vector $v_\mu \in V(\mu)$. There is an element $p_\mu \in \Uqg$ such that 
	\begin{enumerate}[(i)]
		\item $p_\mu(v_\mu) = v_\mu$.
		\item For any $\nu \neq \mu$, $p_\mu$ kills the weight space $V_\nu \subseteq V(\mu)$.
		\item $p_\mu V(\nu) = 0$ unless $\langle \nu - \mu, \alpha_i^\vee \rangle \geq 0$ for some $\alpha_i^\vee$.
	\end{enumerate}
\end{lem}
\begin{rmk}
	Our statement of $(iii)$ differs slightly from the one in \cite{kamnitzer_tingley_2008}, so we provide a proof.
\end{rmk}
\begin{proof}
	Fix a lowest weight vector $v_\mu^{\mathrm{low}} \in V(\mu)$. Recall that $V(\mu) = U_q^- \cdot v_\mu$, so there exists $F \in U_q^-$ such that $v_\mu^{\mathrm{low}} = F v_\mu$. Similarly, there exists $E \in U_q^+$ such that $E v_\mu^{\mathrm{low}} = v_\mu$. It follows that $p' = EF$ satisfies $(i)$ and $(ii)$, since for every non-zero weight space $V_\nu \subseteq V(\mu)$ we must have $\nu < \mu$, so that $F V_\nu = 0$.
	
	Now write $\mu = \sum_j \mu_j \omega_j$, set $p_i = (E_i)^{\mu_i} (F_i)^{\mu_i}$, and define 
\[
	p_\mu' = \bigg(\prod_{i=1}^r p_i \bigg) p',
\]
	where the product is taken in any order. Now if $\langle \nu - \mu, \alpha_i^\vee \rangle < 0$ for each $\alpha_i^\vee$, then $p_i$ acts by zero on $V(\nu)$. Conversely, if $\langle \nu - \mu, \alpha_i^\vee \rangle \geq 0$ for some $\alpha_i^\vee$, there is a weight vector in $V(\nu)$ that is not killed by $p_\mu'$ because the simple reflection $s_i$ corresponding to the simple root $\alpha_i$ acts as
\[
	s_i(\nu) = \nu - \langle \nu, \alpha_i^\vee \rangle \alpha_i.
\]
	We conclude that some scalar multiple of $p_\mu'$ satisfies $(i)$-$(iii)$ because each $p_i$ fixes $V_\mu \subseteq V(\nu)$. 
\end{proof}

\begin{proof}[Proof of \Cref{equivalence of completions}]
	Lemma~3.5 in \cite{kamnitzer_tingley_2008} shows that if $I(\mu)$ denotes the kernel of the action of $\Uqg$ on $V(\mu)$, then $\Uqg / I(\mu)$ is isomorphic to $\End_{\mathbb{C}(q)} V(\mu)$. Combined with \Cref{defn p lambda}, this means we can realize any endomorphism of $V(\mu)$ using an element of $\Uqg$ that kills every other $V(\nu)$ unless $\langle \nu - \mu, \alpha_i^\vee \rangle \geq 0$ for some $i = 1, \ldots, r$. Thus the completion defined by \labelcref{kt completion of Uq} is equivalent to the direct product $M = \displaystyle \prod_\mu \End_{\mathbb{C}(q)} (V(\mu))$. 
	
	The equivalence between $M$ and the completion defined by \Cref{completion of Uq} follows similarly. Each element of the direct limit has a well-defined action on every $V(\mu)$, so there is a map taking elements of the direct limit into $M$. Injectivity follows from the definition of the direct limit and surjectivity follows from the technical \Cref{defn p lambda}.
\end{proof}

\Cref{equivalence of completions} implies $\widetilde{\Uqg}$ is equivalent to the completion of $\Uqg$ in the weak topology generated by the matrix coefficients of all finite-dimensional $\Uqg$-modules \cite{kamnitzer_tingley_2008}. In particular, $\Uqg$ is a locally convex t.v.s. with topology generated by the family of seminorms
\[
	p^\mu_{i, j}(x) = | v^i(x\rhd v_j) |, \quad \text{where} \quad i, j = 1, \ldots, \dim(V(\mu)), \text{ and } \mu \in \mathcal{P}^+.
\]
Here $v_i$ is a weight basis of $V(\mu)$ and $v^i$ is the corresponding dual basis of $V(\mu)^*$. This topology is Hausdorff: Theorem $7.13$ in \cite{klimyk_schmudgen_1997} guarantees no non-zero element of $\Uqg$ acts by zero on every $V(\mu)$, which means
\[
	\bigcap_{\mu, i, j} \ker (p_{i,j}^\mu) = 0.
\]
The topology is generated by countably many seminorms, so sequences characterize completeness. Each equivalence class of Cauchy sequences on $\Uqg$ induces an endomorphism on each $V(\mu)$ considered with its weak topology and \Cref{equivalence of completions} proves that each such family of endomorphisms corresponds uniquely to a formal sum $\sum_k x_k$ with $x_k \in \Uqg$ such that $x_k \rhd V(\mu) = 0$ for all but finitely many $k$. 

Thus $\widetilde{\Uqg}$ is a topological algebra in the sense of \Cref{topological alg}. The Hopf algebra structure of $\Uqg$ naturally extends to the completion $\widetilde{\Uqg}$. However, the comultiplication on $\widetilde{\Uqg}$ is only a \textit{topological} comultiplication, in the sense that $\Delta_q(\widetilde{\Uqg})$ is not contained in $\widetilde{\Uqg} \otimes \widetilde{\Uqg}$, but rather in the completed tensor product $\widetilde{\Uqg} \widetilde{\otimes} \widetilde{\Uqg}$. This is most easily seen using \Cref{equivalence of completions}: each element of $\widetilde{\Uqg}$ is equivalent to a choice of $f \in \End(V(\mu))$ for each $\mu \in \mathcal{P}^+$, so $\Delta_q$ maps $\widetilde{\Uqg}$ into $\displaystyle \prod_{\mu, \nu} \End_{\mathbb{C}(q)}(V(\mu)) \otimes \End_{\mathbb{C}(q)}(V(\nu))$, which is understood as the completion of $\displaystyle \prod_\mu \End_{\mathbb{C}(q)}(V(\mu)) \otimes \prod_\nu \End_{\mathbb{C}(q)}(V(\nu))$ \cite{kamnitzer_tingley_2008}. 

It follows that specifying an element of $\widetilde{\Uqg} \widetilde{\otimes} \widetilde{\Uqg}$ is equivalent to specifying an action on every tensor product of any pair of irreducible $\Uqg$-modules. This means $\mathcal{E}_h$ and $\mathcal{S}_h$ define elements $\widetilde{\mathcal{E}_q}$ and $\widetilde{\mathcal{S}_q}$ in $\widetilde{\Uqg} \widetilde{\otimes} \widetilde{\Uqg}$. Therefore $\widetilde{\Rmat_q} = \widetilde{\mathcal{E}_q}\widetilde{\mathcal{S}_q}$ is a well-defined element of $\widetilde{\Uqg} \widetilde{\otimes} \widetilde{\Uqg}$. This element must be an $R$-matrix of $\widetilde{\Uqg}$ because it acts as the braiding on $\Uqgmodcat$ on any tensor product of irreducibles. Hence $\widetilde{\Uqg}$ is a braided topological bialgebra.

Moreover, there is a ribbon element in $\widetilde{\Uqg}$, in the sense of \Cref{ribbon algebra}. Concretely, the ribbon element $\theta_h$ defined in \Cref{R matrix of Uh special relation} induces an action on each $\Uqg$-module via \Cref{ribbon elt action} once we identify $q = e^h$. In view of \Cref{equivalence of completions}, this means $\theta_h$ specifies an element $\theta_q$ in $\widetilde{\Uqg}$. 

\begin{thm}\label{Uq is a ribbon alg}
	For any $V(\mu) \in \Uqgmodcat$, let $\theta_{V(\mu)}: V(\mu) \to V(\mu)$ denote the isomorphism acting as the scalar $q^{\langle \mu, \mu + 2\rho\rangle}$. Then $\theta_{V(\mu)}$ is a family of twists in the sense of \Cref{ribbon algebra}, making $\Uqgmodcat$ into a ribbon category. The twists are induced by the action of a single element $\inv{\theta_q} \in \widetilde{\Uqg}$, which means $\widetilde{\Uqg}$ is a topological ribbon algebra.
\end{thm}
\begin{proof}
	\Cref{equivalence of completions} implies the family of maps $(\theta_{V(\mu)})_{\mu \in \mathcal{P}}$ specifies a single element $\inv{\theta}_q \in \widetilde{\Uqg}$. Recalling that $\rho = \sum_i \mu_i \alpha_i^\vee$, set $K_{2\rho} = \prod_i K_i^{2\inv{d_i} \mu_i}$, with the product taken in any order. As the sum of the positive roots, $2\rho \in \mathcal{Q}$, and therefore each $2 \inv{d_i} \mu_i \in \mathbb{Z}$. The element $K_{2\rho}$ acts as the scalar $q^{\langle \mu, 2\rho \rangle}$ on the weight space $V_\mu$ of any finite-dimensional module. Arguing as in the proof of \Cref{R matrix of Uh special relation}, we construct a distinguished element $u \in \widetilde{\Uqg}$, this time using the $R$-matrix $\Rmat_q$. The eigenvalue calculation in that proof implies that the element $K_{2\rho} \inv{u} \in \widetilde{\Uqg}$ also acts by the scalar $q^{\langle \mu, \mu + 2\rho\rangle}$ on $V(\mu)$. The isomorphism of \Cref{equivalence of completions} guarantees that $\inv{\theta} = K_{2\rho} \inv{u}$. \Cref{R matrix of Uh special relation} then shows that $(\Rmat_q)_{21}\Rmat_q$ and $\Delta(\theta)(\inv{\theta} \widetilde{\otimes} \inv{\theta})$ have the same action on the tensor product of any pair of finite-dimensional $\Uqg$-modules, so they define the same element of $\widetilde{\Uqg} \widetilde{\otimes} \widetilde{\Uqg}$ by \Cref{equivalence of completions}. The remaining relations $S_q(\theta) = \theta$ and $\varepsilon_q(\theta) = 1$ follow from the argument used to establish \Cref{R matrix of Uh special relation}; see for more details \cite[p.~323]{chari_pressley_1994}.
\end{proof}
\begin{rmk}
	The element $\theta_q$ can also be constructed in the PBW-type basis as a sum of canonical elements analogous to the Casimir element of $\Ug$, by defining a non-degenerate bilinear pairing between $U_q^+$ and $U_q^-$. For details, see \cite[Section~6.3.3]{klimyk_schmudgen_1997}.
\end{rmk}	 	


\bibliographystyle{alphaurl}
\bibliography{ref}
	
\end{document}